\documentclass[11pt,a4paper]{article}
\usepackage{amsbsy}
\usepackage{amssymb}
\usepackage{graphicx}
\usepackage{amsmath}
\usepackage{epsfig}
\usepackage{amsfonts}
\usepackage{theorem}

\newtheorem{definition}{Definition}
\newtheorem{thm}{Theorem}

\newtheorem{lemma}[thm]{Lemma}
\newtheorem{remark}{Remark}
\newtheorem{corollary}{Corollary}
\newtheorem{proposition}{Proposition}

\input{tcilatex}

\begin{document}

\title{\bfseries {\normalsize NON-CENTRAL LIMIT THEOREMS FOR RANDOM FIELDS SUBORDINATED TO GAMMA-CORRELATED
RANDOM FIELDS}}
\author{{\normalsize N.N. Leonenko, M.D. Ruiz-Medina and M.S. Taqqu \thanks{%
Correspondence to: M. D. Ruiz-Medina, Department of Statistics and
Operations Research, University of Granada, Campus de Fuente Nueva
s/n,
E-18071 Granada, Spain.} \thanks{%
E-mail: {\ttfamily mruiz@ugr.es}}}}

\date{January 11, 2015}
\maketitle

\begin{abstract}
A reduction theorem is proved for functionals of Gamma-correlated
random fields with long-range dependence in $d$-dimensional space.
In the particular case of  a non-linear  function of a chi-squared
random field  with  Laguerre rank  equal to one, we apply the
Karhunen-Lo\'eve expansion and the Fredholm determinant formula   to
obtain the characteristic function of its Rosenblatt-type limit
distribution. When the  Laguerre rank  equals  one  and  two, we
obtain
 the multiple Wiener-It$\hat{\mbox{o}}$ stochastic integral representation of the limit distribution.
 In  both cases, an infinite series representation   in terms of independent  random
variables is  constructed for the limit random variables.

\medskip

\end{abstract}

\section{Introduction}
This paper considers the family of Gamma-correlated random fields
within the general class of \emph{Lancaster-Sarmanov random fields}.
Such a class includes non-Gaussian random fields with given marginal
distributions and given covariance structure. The bivariate
densities of these fields have diagonal expansions.  Lancaster
(1958) and   Sarmanov (1963) idenpendently discovered these
expansions in the context of Markov processes, namely, for dimension $d=1,$
and  correlation function  $\gamma (|x-y|)=
\exp\left(-c|x-y|\right),$ $c>0.$ This line of research was also
continued   by Wong and Thomas (1962), where Laguerre polynomials
were used as well as Hermite and Jacoby polynomials, in Markovian
settings. The extension of these limit theorems, based on  bilinear
expansions, to the context of long-range dependent (LRD) processes
was considered in Berman (1982,1984), and also for random fields by
 Leonenko
(1999), Anh and Leonenko (1999) and Anh, Leonenko and Ruiz-Medina
(2013), among others. That this class of random fields is not empty
 follows from the results by Joe (1997), who constructed the
system of finite-dimensional distributions for a given bivariate
distribution consistent with their marginal distributions, using the
calculus of variations and the maximum entropy principle. Some
properties of stationary sequences with bivariate densities having
diagonal expansions, and their limit theorems were obtained by Gajek
and Mielniczuk (1999) and Mielniczuk (2000). Specifically, in Gajek
and Mielniczuk (1999), long-range dependence sequences $\{Z_{i}\}_{i=1}^{\infty }$
 with exponential marginal distributions and its subordinated sequences are studied. In particular, processes of the form
$Z_{i}=(X_{i}^{2}+Y_{i}^{2})/2,$ $i=1,2\dots,$ where $\{X_{i}\}_{i=1}^{\infty }$ and $\{Y_{i}\}_{i=1}^{\infty }$ are independent copies of a zero-mean stationary Gaussian process with long-range dependence, are investigated. The asymptotic behaviour of a partial-sum process of the long-range-dependent sequence
$\{G(Z_{i})\}_{i=1}^{\infty },$ constructed by subordination from $\{Z_{i}\}_{i=1}^{\infty },$ is the same as that of
the first nonvanishing term of its Laguerre expansion (see also Taqqu, 1975, 1979, in relation to central and noncentral limit theorems for long-range dependence  processes in discrete time).
In Mielniczuk (2000), different properties  of bivariate  densities (not necessarily associated with stochastic processes)
are studied in the case where they admit a diagonal expansion, which is referred as
Lancaster-Sarmanov expansion, including Mehler's formula for bivariate Gaussian distributions,
 Myller--Lebedev
or Hille--Hardy formula for bivariate Gamma distributions, among others (see, for example,  Bateman and Erdelyi, 1953, Chapter 10). In particular,
  Mehler's equality and Gebelein's inequality are generalized. In addition,  conditions are established for defining long-range dependence sequences
  satisfying the reduction principle, by subordination to discrete time stationary processes. The present  paper extends these results to the general setting
  of random fields with continuous $d$-dimensional parameter space, defined by a regular compact domain of $\mathbb{R}^{d}.$ In particular, a reduction
  theorem is derived for Gamma-correlated random fields with long-range dependence. Some noncentral limit results are established for long-range dependence
  random fields constructed by subordination from chi-squared random fields, in the cases of function $G$ having Laguerre rank equal to one and two.
We will pursue in more details  the chi-squared random field case
where an explicit representation of the random field is available.

The paper is organized as follows. In Section \ref{Sec2}, we define the Lancaster-Sarmanov fields.
In Section \ref{sLaguerre}, we consider the case of Gamma and chi-squared random fields. In Section \ref{reductp}, we prove
the reduction principle for Gamma-correlated random fields. In Section \ref{secsr},
 limit theorems are obtained for the case  of functions of chi-squared random fields with Laguerre rank equal to one and two. We give
a multiple Wiener-It$\hat{o}$ stochastic integral representation of
the limits. Infinite series representations of the  limits obtained
in Section \ref{secsr} are obtained in Section \ref{ISR}.  We
establish  infinite divisibility in Section \ref{FC}.

\section{The Lancaster-Sarmanov random fields}
\label{Sec2}
We now introduce here the class of Lancaster-Sarmanov random fields
with given one-dimensional marginal distributions and general
covariance structure. Denote by $\mathcal{L}_{2}(\Omega
,\mathcal{F},P)$ the Hilbert space of zero-mean second-order random
variables defined on the complete probability space $(\Omega
,\mathcal{F},P).$
For a probability density function $p$ on the interval $(l,r),$ with $%
-\infty \leq l<r\leq \infty ,$ we consider the Hilbert space $L^{2}((l,r),~p(u)du)$ of equivalence classes of Lebesgue measurable functions
$h:$ $(l,r)\rightarrow \mathbb{R}$ satisfying

\begin{equation*}
\int_{l}^{r}h^{2}(u)\ p(u)\ du<\infty ,\quad p(u)\geq 0.
\end{equation*}%
Let us also consider a complete orthonormal system $\{e_{k}(u)\}_{k=0
}^{\infty }$ of functions in \linebreak
$L^{2}((l,r),~p(u)du)$, that is,
\begin{equation}
\int_{l}^{r}e_{k}(u)\ e_{m}(u)\ p(u)du=\delta _{k,m},  \label{2.1}
\end{equation}%
where $\delta _{k,m}$ denotes the Kronecker delta function. We introduce
the following condition:

\bigskip

\noindent {\bfseries Condition A0} Let $\{\xi (\mathbf{x}),\ \mathbf{x}\in \mathbb{R}^{d}\}$
be a mean-square continuous zero-mean  homogeneous isotropic random field with
correlation function

\begin{equation*}
\gamma (\left\Vert \mathbf{x}\right\Vert )=\frac{B(\left\Vert \mathbf{x}\right\Vert )}{B(0)}%
,\quad B(\left\Vert \mathbf{x}\right\Vert )=\mathrm{Cov}(\xi (0),\xi (\mathbf{x})),\quad
\mathbf{x}\in \mathbb{R}^{d}.
\end{equation*}
We assume that the densities
\begin{eqnarray*}
p(u) &=&\frac{d}{du}P\{\xi (\mathbf{x})\leq u\},~u\in (l,r), \\
p(u,w,\left\Vert \mathbf{x}-\mathbf{y}\right\Vert )
&=&\frac{\partial ^{2}}{\partial
u\partial w}P\{\xi (\mathbf{x})\leq u,\ \xi (\mathbf{y})\leq w\}, \quad (u,w)\in (l,r)\times (l,r),
\end{eqnarray*}%
\noindent exist, and that the bilinear expansion
\begin{equation}
p(u,w,\left\Vert \mathbf{x}-\mathbf{y}\right\Vert )=p(u)\ p(w)\ \left( 1+\sum_{k=1}^{\infty
}\gamma ^{k}(\left\Vert \mathbf{x}-\mathbf{y}\right\Vert )\ e_{k}(u)\ e_{k}(w)\right)
\label{2.2}
\end{equation}%
\noindent holds, where
\begin{equation*}
\sum_{k=1}^{\infty }\gamma ^{2k}(\left\Vert \mathbf{x}\right\Vert )<\infty, \quad
\forall \left\Vert \mathbf{x}\right\Vert >0,
\end{equation*}%
\noindent and $\{e_{k}(u)\}_{k=0}^{\infty }$ is, as before,  a
complete orthonornal system in the Hilbert space \linebreak
$L^{2}((l,r),p(u)du).$ Assume also that  $e_{0}(u)\equiv 1.$ The
symmetric kernel \begin{equation}K(u,w,\|\mathbf{x}-\mathbf{y}\|)=\frac{p(u,w,
\left\Vert \mathbf{x}-\mathbf{y}\right\Vert )}{p(u) p(w)}= 1+\sum_{k=1}^{\infty
}\gamma^{k}(\|\mathbf{x}-\mathbf{y}\|)e_{k}(u)e_{k}(w)\label{e:kernel}
\end{equation}
\noindent plays an important role.

 The series (\ref{2.2}) converges
in the mean-square sense if the integral
\begin{eqnarray}
I^{2}&=&\int_{l}^{r}\int_{l}^{r}K^{2}(u,w,\|\mathbf{x}-\mathbf{y}\|)p(u)p(w)dudw \nonumber\\
&=&\int_{l}^{r}\int_{l}^{r}K^{2}(u,w,\|\mathbf{x}-\mathbf{y}\|)dP\left\{
\xi (\mathbf{x})\leq u\right\} dP\left\{ \xi (\mathbf{y})\leq
w\right\} <\infty ,\nonumber
\end{eqnarray}
\noindent where  $I^{2}-1$ is known as the Pearson functional for
the bivariate density $p(u,w,\left\Vert \mathbf{x}\right\Vert )$ (see, for example, Lancaster, 1963).
Then, the symmetric kernel $K(u,w)$ belongs to the product space
$L^{2}((l,r)\times (l,r), p\otimes p(u,w)dudw)$   of square integrable functions on $(l,r)\times (l,r),$ with respect to the measure
$p\otimes p(u,w)dudw.$ Thus, the kernel $K$ defines an integral
Hilbert-Schmidt operator on the space  $L^{2}((l,r),p(u)du).$ From
the spectral theorem for compact and self-adjoint operators (see,
for example, Dautray and Lions, 1985, p.112), for each $\mathbf{x},$
$\mathbf{y}\in  \mathbb{R}^{d},$ the kernel $K$ admits the diagonal
spectral expansion
\begin{equation}
K(u,w)=\frac{p(u,w,\left\Vert \mathbf{x}-\mathbf{y}\right\Vert )}{p(u)\ p(w)}%
=\sum_{k=0}^{\infty }r_{k}(\left\Vert \mathbf{x}-\mathbf{y}\right\Vert )e_{k}(u)e_{k}(w),
\label{2.3}
\end{equation}%
\noindent where convergence holds in the space $L^{2}((l,r)\times
(l,r), p\otimes p(u,w)dudw).$ Here, $r_{k}\left(\left\Vert
\mathbf{x}-\mathbf{y}\right\Vert\right)$ is the sequence of
eigenvalues, associated with the orthonormal system of
eigenfunctions $\{e_{k}(u)\}_{k=0}^{\infty },$    which could  also
depend on $\mathbf{x}$ and $\mathbf{y}$ in a general setting.

\medskip

\noindent Thus,  {\bfseries Condition A0} postulates the expansion (\ref{2.3}) for the case where
$$r_{k}(\mathbf{x},\mathbf{y})=\gamma ^{k}(\left\Vert \mathbf{x}-\mathbf{y}\right\Vert ),$$ \noindent and $e_{k}(u)$ does not
depend on $\mathbf{x}$ and $\mathbf{y}.$ {\bfseries Condition A0} then implies
\begin{eqnarray}
& &\mathrm{E}[e_{k}(\xi (\mathbf{x}))]=\int_{l}^{r}e_{k}(u)p(u)du=0,\quad k\geq 1\nonumber\\
& & \mathrm{E}[e_{n}(\xi (\mathbf{x}))e_{m}(\xi (\mathbf{y}))] =
\int_{l}^{r} \int_{l}^{r}
e_{n}(u)e_{m}(w)p(u,w,\|\mathbf{x}-\mathbf{y}\|)dudw
\nonumber\\
& &=\int_{l}^{r} \int_{l}^{r}e_{n}(u)e_{m}(w)p(u)p(w)\left(
1+\sum_{k=1}^{\infty
}\gamma^{k}(\|\mathbf{x}-\mathbf{y}\|)e_{k}(u)e_{k}(w)\right) dudw\nonumber\\
& &=
 \delta _{n,m}\ \gamma ^{n}(\left\Vert
\mathbf{x}-\mathbf{y}\right\Vert ),\quad  n,m\geq 1. \label{2.4}
\end{eqnarray}
We will call the random fields satisfying {\bfseries Condition A0}
\emph{Lancaster-Sarmanov random fields}, due to Lancaster (1958) and
Sarmanov (1963).  In the next section, we will refer to the special
case of Gamma-correlated random fields, and, in particular, to the
case of chi-squared random fields. We will also let $(l,r)$ in (\ref{2.1}) be $(0,\infty ).$

\section{Gamma-correlated random fields}

\label{sLaguerre}
In this paper
all random fields considered are assumed to be measurable and
mean-square continuous.
 We refer to the class of random fields with Gamma marginal
distribution and given correlation function. For details see Berman
(1982,1984), Leonenko (1999), Anh, Leonenko and Ruiz-Medina (2013),
among others. Following the ideas of Lancaster (1958) and   Sarmanov
(1963), we introduce a homogeneous and isotropic random field
$\{\xi(\mathbf{x}),\ \mathbf{x}\in \mathbb{R}^{d}\},$ with given
one-dimensional Gamma distributions, and given correlation structure
$\gamma
(\|\mathbf{x}-\mathbf{y}\|)=\mbox{Corr}\left(\xi(\mathbf{x}),\xi(\mathbf{y})\right),$
$\mathbf{x},\mathbf{y}\in \mathbb{R}^{d}.$
Let
\begin{equation}
p_{\beta }(u)=\frac{1}{\Gamma (\beta )}u^{\beta-1}\exp(-u),\quad
u>0,\quad \beta >0, \label{gammad}
\end{equation}
\noindent be a  Gamma density, and let $L^{2}((0,\infty),p_{\beta
}(u)du)$ be the Hilbert space of square integrable functions with
respect to the measure $p_{\beta }(u)du,$ i.e., the space of
functions $F$ such that \begin{equation}\int_{0}^{\infty
}F^{2}(u)p_{\beta }(u)du<\infty.\label{sigamma}\end{equation}

An orthogonal basis of the Hilbert space $L^{2}((0,\infty),p_{\beta
}(u)du)$ can be constructed  from generalized Laguerre polynomials
$L_{k}^{(\beta )},$ $k\geq 0,$ of index  $\beta $ (see Bateman and
Erdelyi, 1953). Specifically, its  elements are defined as follows:
For $k,m\geq 0,$
\begin{eqnarray}
 e_{k}(u)=e_{k}^{(\beta )}(u)=L_{k}^{(\beta
-1)}(u)\left[\frac{k!\Gamma (\beta )}{\Gamma (\beta
+k)}\right]^{1/2}, \quad \int_{0}^{\infty}e_{k}^{(\beta )}(u)\ e_{m}^{(\beta )}(u)\ p_{\beta }(u)du=\delta _{k,m},
\label{glp}
\end{eqnarray}
\noindent where by Rodr\'{\i}guez formula for Laguerre polynomials
\begin{equation}
L_{k}^{(\beta )}=
L_{k}^{\left(\beta\right)}(u)=(k!)^{-1}u^{-\beta }\exp(u)\frac{d^{k}%
}{du^{k}}\left\{\exp(-u)u^{\beta +k}\right\}.
 \label{glp2} \end{equation}
 The first three polynomials are then given by
\begin{eqnarray}
e_{0}^{(\beta )}(u)&\equiv & 1,\quad e_{1}^{\left(\beta
\right)}(u)=\sqrt{\frac{1}{\beta }}\left(\beta
-u\right)\nonumber\\
e_{2}^{\left(\beta\right)}(u)&=&\left(u^{2}-2\left(\beta
+1\right)~u+\left(\beta+1\right)\beta\right)~\left[2\left(
\beta+1\right)\beta\right]^{-1/2}. \label{threepoly}
\end{eqnarray}

Applying Myller-Lebedev or Hille-Hardy formula (see Bateman and
Erdelyi, 1953, Chapter 10) we obtain
\begin{eqnarray}
p_{\beta}(u,w,\|\mathbf{x}-\mathbf{y}\|)&=&p_{\beta
}(u)~p_{\beta}(w)~\left[ 1+\sum_{k=1}^{\infty }\gamma^{k}
(\|\mathbf{x}-\mathbf{y}\|)~e_{k}^{\left(\beta
\right)}(u)~e_{k}^{\left(\beta\right)}(w)\right]\nonumber\\
&=&\left( \frac{uw}{\gamma (\|\mathbf{x}-\mathbf{y}\|)}\right)
^{(\beta-1)/2}\exp \left\{ -\frac{u+w}{1-\gamma
(\|\mathbf{x}-\mathbf{y}\|) }\right\} \nonumber\\
&\times &
 I_{\beta-1}\left( 2\frac{\sqrt{%
uw\gamma (\|\mathbf{x}-\mathbf{y}\|)}}{1-\gamma
(\|\mathbf{x}-\mathbf{y}\|)}\right)\frac{1}{\Gamma \left(\beta
\right)~(1-\gamma (\|\mathbf{x}-\mathbf{y}\|))},  \label{210bb}
\end{eqnarray}
\noindent where $\gamma (\|\mathbf{x}-\mathbf{y}\|)$ is a continuous
non-negative definite kernel on $\mathbb{R}^{d}\times
\mathbb{R}^{d},$ depending on $\|\mathbf{x}-\mathbf{y}\|,$ and
$I_{\varrho }(z)$ is the modified Bessel function of the first kind
of order $\varrho ,$ with
$$I_{\varrho }(z)=\frac{(z/2)^{\varrho}}{\sqrt{\pi}\Gamma \left(\varrho +\frac{1}{2}\right)}\int_{-1}^{1}(1-t^{2})^{\varrho -1/2}\exp(zt)dt,\quad z>0.$$

Summarizing, one can define a homogeneous and isotropic
gamma-correlated random field as  a  random field $\left\{\xi
(\mathbf{x}),\ \mathbf{x}\in \mathbb{R}^{d}\right\},$ such that its one
dimensional densities $$\frac{d}{du}P\left[ \xi(\mathbf{x})\leq
u\right]$$ \noindent and two-dimensional densities $$p\left(u,w,
\|\mathbf{x}-\mathbf{y}\|\right)=\frac{\partial^{2}}{\partial
u\partial w}P\left[ \xi(\mathbf{x})\leq u,\xi(\mathbf{y})\leq
w\right]$$\noindent are defined by (\ref{gammad}) and (\ref{210bb}),
respectively. In addition, the correlation function $\gamma $
satisfies
$$\sum_{k=1}^{\infty }\gamma^{2k}(\|\mathbf{z}\|) <\infty,\quad \|\mathbf{z}\|>0.$$
From equation (\ref{sigamma}), $F(u)$ can be expanded into the
series
\begin{equation}
F(u)=\sum_{q=0}^{\infty }C_{q}^{L}e_{q}^{(\beta )}(u),\quad
C_{q}^{L}= \int_{0}^{\infty }F(u)~e_{q}^{(\beta
)}(u)p_{\beta }(u)~du, \quad q=0,1,2,\ldots,
\label{cgcrf}
\end{equation}
\noindent  which converges in the Hilbert space
$L_{2}((0,\infty ),p_{\beta }(u)du)$. In particular,
\begin{equation}C_{0}^{L}=\int_{0}^{\infty} F(u)e_{0}^{(\beta )}(u)p_{\beta }(u)~du= \mathrm{E}[F(\xi (\mathbf{x}))].\label{e:coL}
\end{equation}

The \emph{Laguerre rank} of the function $F$ is defined as the smallest $%
k\geq 1$ such that
\begin{equation*}
C_{1}^{L}=0,\dots ,C_{k-1}^{L}=0,\ C_{k}^{L}\neq 0.
\end{equation*}
From equation (\ref{210bb}), for a homogeneous and isotropic
Gamma-correlated random field $\{\xi (\mathbf{x}),\ \mathbf{x}\in
\mathbb{R}^{d}\},$ with correlation function $\gamma,$ the following
identities hold:
\begin{equation}
E[e_{k}^{(\beta )}(\xi(\mathbf{x}))]= 0,\quad
E[e_{m}^{(\beta )}(\xi(\mathbf{x}))e_{k}^{(\beta
)}(\xi(\mathbf{y}))]= \delta_{m,k}\gamma^{k}
(\|\mathbf{x}-\mathbf{y}\|).\label{orthLp}
\end{equation}

In order to introduce long-range dependence for Gamma-correlated
random fields, we assume the following condition:

\medskip

\noindent \textbf{Condition A1}. The non-negative definite function
\begin{equation}
\gamma (\|\mathbf{z}\|)=\frac{\mathcal{L}(\Vert
\mathbf{z}\Vert )}{\Vert \mathbf{z}\Vert ^{\delta }},\quad
\mathbf{z}\in \mathbb{R}^{d},\quad 0<\delta < d, \label{corrgammRF}
\end{equation}
\noindent where $\mathcal{L}$ is a slowly varying function at
infinity.

\subsection{The chi-squared random fields}

One can construct
examples of random fields with marginal density (\ref{gammad})
and bivariate probability density (\ref{210bb}) considering the class of
chi-squared random fields. The chi-squared random fields are given by
\begin{equation}
\chi _{r}^{2}(\mathbf{x})=\frac{1}{2}\left(Y_{1}^{2}(\mathbf{x}%
)+\dots +Y_{r}^{2}(\mathbf{x})\right),\quad \mathbf{x}\in
\mathbb{R}^{d}, \label{Chi_1}
\end{equation}%
\noindent where $Y_{1}(\mathbf{x}),\dots,Y_{r}(\mathbf{x})$ are
independent copies of Gaussian random field $\{Y(\mathbf{x}),\
\mathbf{x}\in \mathbb{R}^{d}\}$ with covariance function
$B(\left\Vert \mathbf{x}\right\Vert )$ with $B(\|\mathbf{0}\|)=1.$  In this case
\begin{equation}\gamma
(\|\mathbf{x}-\mathbf{y}\|)=\frac{\mbox{Cov}(\chi^{2}_{r}(\mathbf{x}),\chi^{2}_{r}(\mathbf{y}) )}{\mbox{Var}(\chi^{2}_{r}(\mathbf{0}))}=B^{2}(\|\mathbf{x}-\mathbf{y}\|),\quad \beta =r/2.\label{idcorr}
\end{equation}
\noindent  Note that by construction, the
correlation function of chi-squared random fields is always
non-negative. Moreover,

\begin{equation*}
\mathrm{E}\chi _{r}^{2}(\mathbf{x})=\frac{r}{2},\quad
\mathrm{Var}\chi
_{r}^{2}(\mathbf{x})=\frac{r}{4}\mathrm{Var}\ Y_{1}^{2}(\mathbf{x})=\frac{r}{%
2},\quad \mathrm{Cov}(\chi _{r}^{2}(0),\chi _{r}^{2}(\mathbf{x}))=\frac{r}{2}%
B^{2}(\left\Vert \mathbf{x}\right\Vert ).
\end{equation*}
\noindent and
\begin{equation}
\mathrm{E}[e_{k}^{(r/2)}(\chi _{r}^{2}(\mathbf{x}))~e_{m}^{(r/2)}(\chi _{r}^{2}(\mathbf{y}%
))]=\delta _{m,k}~B^{2m}(\left\Vert \mathbf{x}-\mathbf{y}\right\Vert
), \label{(1)}
\end{equation}

\noindent since as noted in (\ref{2.1}),
\begin{equation*}
\int_{0}^{\infty }e_{k}^{(r/2)}(u)\ e_{m}^{(r/2)}(u)\ p_{r/2}(u)\ du=\delta _{k,m}.
\end{equation*}

In the case of chi-squared random fields (\ref{Chi_1}) the analogous
of \textbf{Condition A1} setting in (\ref{corrgammRF}) is the following \textbf{Condition A2}.

\medskip

\noindent \textbf{Condition A2}. The random field $\{Y(\mathbf{x}),\ \mathbf{%
x}\in \mathbb{R}^{d}\},$ whose independent copies define the chi-squared random field (\ref{Chi_1}), is a measurable zero-mean Gaussian
homogeneous and
isotropic mean-square continuous random field on a probability space $%
(\Omega ,\mathcal{A},P),$ with $\mathrm{E}Y^{2}(\mathbf{x})=1,$ for all $%
\mathbf{x}\in \mathbb{R}^{d},$ and correlation function $\mathrm{E}[Y(%
\mathbf{x})Y(\mathbf{y})]=B(\Vert \mathbf{x}-\mathbf{y}\Vert )$ of
the form:
\begin{equation}
B(\Vert \mathbf{z}\Vert )=\frac{\mathcal{L}(\Vert \mathbf{z}\Vert
)}{\Vert \mathbf{z}\Vert ^{\alpha }},\quad \mathbf{z}\in
\mathbb{R}^{d},\quad 0<\alpha < d/2.  \label{cov}
\end{equation}%
\noindent From \textbf{Condition A2}, the correlation function $B$
of $Y$ is
continuous. It then follows that $\mathcal{L}(r)=\mathcal{O}(r^{\alpha }),$ $%
r\longrightarrow 0.$

\section{Reduction principle for Gamma-correlated random fields}
\label{reductp}
The following reduction principle is an analogous in spirit to the reduction
principle of Taqqu (1975, 1979); for Gamma-correlated random fields,
see also Berman (1982, 1984), Leonenko (1999), among  others.

From equation (\ref{orthLp}),
\begin{equation}
\mathrm{E}\left[\int_{\mathcal{D}(T)}\int_{\mathcal{D}(T)}e_{k}^{(\beta
)}(\xi(\mathbf{x}))e_{m}^{(\beta
)}(\xi(\mathbf{y}))~d\mathbf{x}d\mathbf{y}\right]=\delta
_{k,m}~\newline \sigma _{k}^{2}(T), \label{vargammaRF}
\end{equation}

\noindent where $\mathcal{D}\left( T\right) $ denotes a homothetic transformation of a set $\mathcal{D}\subset
\mathbb{R}^{d}$ with center at the point $\mathbf{0}\in \mathcal{D}$ and coefficient
or scale factor $T>0.$ In addition, $\mathcal{D}$ is assumed to be a
regular compact domain, whose interior has positive Lebesgue measure, and
with boundary having null Lebesgue measure. Dirichlet-regularity here is understood in the general setting
established, for example, by
Fuglede (2005, p. 253), as given in the following definition.
\begin{definition}
\label{def00}
 For a
connected bounded open domain $\mathcal{D}$ with boundary $\partial \mathcal{D}$ we say
that $\mathbf{x}_{0}\in
\partial \mathcal{D}$ is regular if and only if it has a Green kernel $G^{\mathcal{D}}$
such that, for each $\mathbf{x}\in \mathcal{D},$
\begin{equation}
\lim_{\mathbf{x}\rightarrow
\mathbf{x}_{0}}G^{\mathcal{D}}(\mathbf{x},\mathbf{y})=0,\quad \forall
\mathbf{y}\in \mathcal{D}. \label{eqrddgs}
\end{equation}

The set $\mathcal{D}$ is regular if every point of $\partial \mathcal{D}$ is regular.
\end{definition}

Dirichlet regularity of domain $\mathcal{D}$ ensures that the eigenvectors of the operator $K_{\alpha },$ introduced in equation (\ref{RKD}) below,  vanish continuously in the boundary of domain $\mathcal{D}$ (see, for example, Brelot,  1960, p. 137 and Theorem 32, in the context of potential theory, and, more recently,  Chen \emph{et al.}, 2012, p.484, for $0<\alpha <2,$ in the context of  subordinate
processes in domains).

In equation (\ref{vargammaRF}), under {\bfseries Condition A1}, for $0<\delta < d/k,$
\begin{eqnarray}
\sigma _{k}^{2}(T) &=&\mathrm{Var}\left[
\int_{\mathcal{D}(T)}e_{k}^{(\beta )}(\xi
(\mathbf{x}))~d\mathbf{x}\right]  \notag  \label{var22} \\
&=&\int_{\mathcal{D}(T)}\int_{\mathcal{D}(T)}\gamma^{k}\left(\left\Vert \mathbf{x}-%
\mathbf{y}\right\Vert \right)d\mathbf{x}d\mathbf{y}=[a_{d,k}(\mathcal{D}%
)]^{2}T^{2d-k\delta }\mathcal{L}^{k}(T)(1+o(1)),
\label{vargammc}
\end{eqnarray}%
\noindent as $T\longrightarrow \infty ,$ with
\begin{equation}
a_{d,k}(\mathcal{D})=\left[ \int_{\mathcal{D}}\int_{\mathcal{D}}\frac{1}{%
\Vert \mathbf{x}-\mathbf{y}\Vert ^{k\delta
}}d\mathbf{x}d\mathbf{y}\right] ^{1/2},\quad k\geq 1. \label{eqnck}
\end{equation}

 Note that, for the particular case of chi-squared random fields we have from
(\ref{(1)})
\begin{equation*}
\mathrm{E}\left[\int_{\mathcal{D}(T)}\int_{\mathcal{D}(T)}e_{k}(\chi
_{r}^{2}(\mathbf{x}))e_{m}(\chi
_{r}^{2}(\mathbf{y}))~d\mathbf{x}d\mathbf{y}\right]=\delta _{k,m}~\newline
\sigma _{k}^{2}(T),
\end{equation*}%
\noindent where, under {\bfseries Condition A2}, for $0<\alpha <\frac{d}{2k},$

\begin{eqnarray}
\sigma _{k}^{2}(T) &=&\mathrm{Var}\left[
\int_{\mathcal{D}(T)}e_{k}(\chi
_{r}^{2}(\mathbf{x}))~d\mathbf{x}\right]  \notag  \label{var22} \\
&=&\int_{\mathcal{D}(T)}\int_{\mathcal{D}(T)}B^{2k}(\left\Vert \mathbf{x}-%
\mathbf{y}\right\Vert )d\mathbf{x}d\mathbf{y}=[a_{d,k}^{\chi_{r}^{2}}(\mathcal{D}%
)]^{2}T^{2d-2k\alpha }\mathcal{L}^{2k}(T)(1+o(1)),  \notag \\
&&
\end{eqnarray}%
\noindent as $T\longrightarrow \infty ,$ with
\begin{equation}
a_{d,k}^{\chi_{r}^{2}}(\mathcal{D})=\left[ \int_{\mathcal{D}}\int_{\mathcal{D}}\frac{1}{%
\Vert \mathbf{x}-\mathbf{y}\Vert ^{2k\alpha
}}d\mathbf{x}d\mathbf{y}\right] ^{1/2},\quad k\geq 1.
\label{eqnck}
\end{equation}

The following theorem states the reduction principle.

\begin{theorem}
\label{rep} Let $\{\xi (\mathbf{x}),\ \mathbf{x}\in
\mathbb{R}^{d}\}$ be a Gamma-correlated random field.
 Assume that {\bfseries Condition A1} holds, and that the function $F\in L^{2}((0,\infty
), p_{\beta }(u)du)$ has generalized Laguerre rank equal to $k,$
where $p_{\beta}(u)$ is defined by (\ref{gammad}). If  the limiting distribution of the functional  \begin{equation}S_{T}^{L}=\frac{1}{a_{d,k}(\mathcal{D})\mathcal{L}^{k/2}(T)T^{d-(k\delta)/2}}\left[%
\int_{\mathcal{D} (T)}F(\xi
(\mathbf{x}))d\mathbf{x}-C_{0}^{L}T^{d}|\mathcal{D}|\right],\label{eqfglr}
\end{equation}
\noindent for $0<\delta  <d/k$, exists as
$T\longrightarrow \infty ,$ then it
 coincides with the limit distribution of the random
variable
\begin{equation*}
\frac{C_{k}^{L}}{a_{d,k}(\mathcal{D}
)\mathcal{L}^{k/2}(T)T^{d-(k\delta)/2}}\int_{\mathcal{D}
(T)}e_{k}^{(\beta )}(\xi(\mathbf{x}))~d\mathbf{x}.
\end{equation*}
The constants $C_{k}^{L}$ and $C_{0}^{L}$ are defined in equations (\ref{cgcrf}) and (\ref{e:coL}), respectively.
\end{theorem}

\begin{proof}
The proof is based on the generalized  Laguerre polynomial expansion
of the
function $F.$ Specifically, under  {\bfseries Condition A1}, since $\gamma (\left\Vert \mathbf{x}\right\Vert )\leq 1,$ and $%
\gamma (0)=1,$ we have
\[
\gamma^{k+l}(\left\Vert \mathbf{x}\right\Vert )\leq
\gamma^{k+1}(\left\Vert \mathbf{x}\right\Vert ),\quad l\geq 2.
\]%
Hence, from equation (\ref{vargammc}), for $T$  sufficiently large,
\[
E\left[ \frac{1}{a_{d,k}(\mathcal{D}
)\mathcal{L}^{k/2}(T)T^{d-(k\delta )/2}}\left(\int_{\mathcal{D}
(T)}F(\xi (\mathbf{x}))~d\mathbf{x}-C_{0}^{L}T^{d}\left\vert
\mathcal{D} \right\vert -C_{k}^{L}\int_{\mathcal{D}
(T)}e_{k}^{(\beta )}(\xi (\mathbf{x}))~d\mathbf{x}\right)\right]
^{2}\leq
\]%
\[
= \left[ \frac{1}{a_{d,k}(\mathcal{D}
)\mathcal{L}^{k/2}(T)T^{d-(k\delta)/2 }}\right]
^{2}\sum_{j=k+1}^{\infty }(C_{j}^{L})^{2}\int_{\mathcal{D}
(T)}\int_{\mathcal{D} (T)}\gamma ^{j}(\left\Vert
\mathbf{x}-\mathbf{y}\right\Vert )d\mathbf{x}d\mathbf{y}\leq
\]%
\[
\leq \left[ \frac{1}{a_{d,k}(\mathcal{D}
)\mathcal{L}^{k/2}(T)T^{d-(k\delta )/2}}\right]^{2}\int_{\mathcal{D}
(T)}\int_{\mathcal{D} (T)}\gamma^{k+1}(\left\Vert
\mathbf{x}-\mathbf{y}\right\Vert
)d\mathbf{x}d\mathbf{y}\sum_{j=k+1}^{\infty
}(C_{j}^{L})^{2}=K_{R}.
\]

By {\bfseries Condition A1}, for any $\epsilon >0,$ there exists
$A_{0}>0,$ such that for$\left\Vert \mathbf{x}-\mathbf{y}\right\Vert
>A_{0},$
$\gamma (\left\Vert \mathbf{x}-\mathbf{y}\right\Vert )<\epsilon.$ Let $K_{1}=\{(\mathbf{x},\mathbf{y})\in \mathcal{D}(T):\left\Vert
\mathbf{x}-\mathbf{y}\right\Vert \leq A_{0})\},$ and
$K_{2}=\{(\mathbf{x},\mathbf{y})\in \mathcal{D}(T):\left\Vert
\mathbf{x}-\mathbf{y}\right\Vert
>A_{0})\}.$
Then,

\begin{eqnarray}
\int_{\mathcal{D} (T)}\int_{\mathcal{D} (T)}\gamma^{k+1}(\left\Vert
\mathbf{x}-\mathbf{y}\right\Vert )d\mathbf{x}d\mathbf{y} &=&
\left\{\int \int_{K_{1}}+\int
\int_{K_{2}}\right\}\gamma^{k+1}(\left\Vert
\mathbf{x}-\mathbf{y}\right\Vert )d\mathbf{x}d\mathbf{y}\nonumber\\
&=& S_{T}^{(1)}+S_{T}^{(2)}.\label{eqids}
\end{eqnarray}

Using the bound $\gamma^{k+1}(\left\Vert
\mathbf{x}-\mathbf{y}\right\Vert )\leq 1$ on $K_{1},$ and the bound
$\gamma^{k+1}(\left\Vert \mathbf{x}-\mathbf{y}\right\Vert )<\epsilon
\gamma^{k}(\left\Vert \mathbf{x}-\mathbf{y}\right\Vert )$ on
$K_{2},$ we obtain, again, for $T$ sufficiently large,
$$
\left| S_{T}^{(1)}\right| \leq  \left|\int \int_{K_{1}} \gamma
^{k+1}(\left\Vert  \mathbf{x}-\mathbf{y}\right\Vert
)d\mathbf{x}d\mathbf{y} \right|\leq M_{1} T^{d},$$ \noindent for a
suitable constant $M_{1}>0,$ and

\begin{eqnarray}
\left\vert S_{T}^{(2)}\right\vert \leq \left\vert \int
\int_{K_{2}}\gamma ^{k+1}(\left\Vert
\mathbf{x}-\mathbf{y}\right\Vert )d\mathbf{x}d\mathbf{y} \right\vert
\leq \epsilon \left\vert \int \int_{K_{2}}\gamma^{k}(\left\Vert
\mathbf{x}-\mathbf{y}\right\Vert )d\mathbf{x}d\mathbf{y}\right\vert \leq
 \epsilon M_{2}T^{2d-k\delta }\mathcal{L}^{k}(T),\nonumber
\end{eqnarray}
\noindent for suitable $M_{2}>0,$ and arbitrary $\epsilon >0.$
Thus,
\begin{eqnarray}
K_{R} &=&\left[ \frac{1}{a_{d,k}(\mathcal{D}
)\mathcal{L}^{k/2}(T)T^{d-(k\delta
)/2}}\right]^{2}\int_{\mathcal{D} (T)}\int_{\mathcal{D}
(T)}\gamma^{k+1}(\left\Vert \mathbf{x}-\mathbf{y}\right\Vert
)d\mathbf{x}d\mathbf{y}\sum_{j=k+1}^{\infty
}(C_{j}^{L})^{2}\nonumber\\
&\leq & M_{1}\vee M_{2}\left[ \frac{T^{d}}{a_{d,k}^{2}(\mathcal{D}
)\mathcal{L}^{k}(T)T^{2d-k\delta }}+\epsilon \frac{%
T^{2d-k\delta
}\mathcal{L}^{k}(T)}{a_{d,k}^{2}(\mathcal{D}
)\mathcal{L}^{k}(T)T^{2d-k\delta
}}\right] \label{sidd}, \end{eqnarray} which can be made arbitrary
small together with $\epsilon
>0.$

\end{proof}

 The following additional condition is
assumed for the slowly varying function $\mathcal{L}$  in  Theorem \ref{prchf} below.

\medskip

\noindent {\bfseries Condition A3.}  Let $\mathcal{L}$ be the
slowly varying function introduced in {\bfseries Condition A2}.
Assume that, for every $m\geq 2$ there exists a constant $C>0$, such
that
\begin{equation*}
\int_{\mathcal{D}}\dots (m)\dots \int_{\mathcal{D}}\frac{\mathcal{L}(T\Vert \mathbf{x}_{1}-\mathbf{x}%
_{2}\Vert )}{\mathcal{L}(T)\Vert \mathbf{x}_{1}-\mathbf{x}_{2}\Vert
^{\delta }}\frac{\mathcal{L}(T\Vert \mathbf{x}_{2}-\mathbf{x}_{3}\Vert )}{\mathcal{L}%
(T)\Vert \mathbf{x}_{2}-\mathbf{x}_{3}\Vert ^{\delta }}\cdot \cdot
\cdot
\frac{\mathcal{L}(T\Vert \mathbf{x}_{m}-\mathbf{x}_{1}\Vert )}{\mathcal{L}%
(T)\Vert \mathbf{x}_{m}-\mathbf{x}_{1}\Vert ^{\delta }}d\mathbf{x}_{1}d%
\mathbf{x}_{2}\cdot \cdot \cdot d\mathbf{x}_{m}\leq
\end{equation*}%
\begin{equation*}
\leq C\int_{\mathcal{D}}\dots (m)\dots
\int_{\mathcal{D}}\frac{d\mathbf{x}_{1}d\mathbf{x}_{2}\cdot \cdot \cdot
d\mathbf{x}_{m}}{\Vert \mathbf{x}_{1}-\mathbf{x}_{2}\Vert ^{\delta
}\Vert \mathbf{x}_{2}-\mathbf{x}_{3}\Vert ^{\delta }\cdot \cdot
\cdot \Vert \mathbf{x}_{m}-\mathbf{x}_{1}\Vert ^{\delta}}.
\end{equation*}

\bigskip  {\bfseries Condition A3} is satisfied by slowly varying
functions such that
\begin{equation}
\sup_{T,\mathbf{x}_{1},\mathbf{x}_{2}\in \mathcal{D}}\frac{\mathcal{L}(T\Vert \mathbf{x%
}_{1}-\mathbf{x}_{2}\Vert )}{\mathcal{L}(T)}\leq C_{0},
\label{scA2}
\end{equation}
\noindent for $0<C_{0}\leq 1.$ This condition holds, for example,
for logarithmic type slowly
varying functions $\mathcal{L}(\Vert \mathbf{x}\Vert )=\log (C+\Vert \mathbf{%
x}\Vert ),$  $C>0,$ in the case where $\mathcal{D}\subseteq
\mathcal{B}(\mathbf{0}),$
with \linebreak $\mathcal{B}(\mathbf{0})=\{ \mathbf{x}\in \mathbb{R}^{d},\ \|\mathbf{x}%
\|\leq 1\}.$

Note that

\begin{equation*}
B(\Vert \mathbf{z}\Vert )=\frac{1}{(1+\Vert \mathbf{z}\Vert ^{\beta
})^{\gamma }},\quad 0<\beta \leq 2,\quad \gamma >0,
\end{equation*}%
\noindent is a particular case of the family of covariance functions (\ref{cov}) studied here, satisfying   {\bfseries Condition A3}, with $\alpha =  \beta\gamma,$ and
$\mathcal{L}(\Vert
\mathbf{z}\Vert )=\Vert \mathbf{z}\Vert ^{\beta \gamma }/(1+\Vert \mathbf{z}%
\Vert ^{\beta })^{\gamma }.$

\noindent The next result involves chi-squared random fields.  It provides the limit in distribution of
\begin{equation}\frac{1}{a_{d,1}^{\chi_{r}^{2}}(\mathcal{D})\mathcal{L}(T)T^{d-\alpha}}
\int_{\mathcal{D}(T)}e_{1}^{(r/2)}(\chi_{r}^{2}
(\mathbf{x}))d\mathbf{x},\label{chisqrfrg1}
\end{equation}
\noindent and more generally, in view of Reduction Theorem \ref{rep}, of $e_{1}^{r/2}$ being replaced by a function $F$ with Laguerre rank 1.
\begin{theorem}
\label{prchf}
Let $\{\chi _{r}^{2}(\mathbf{x}),\ \mathbf{x}\in \mathbb{R}^{d}\}$ be the chi-squared random field introduced in (\ref{Chi_1}), and  consider
the functional
\begin{equation}S_{T}^{\chi_{r}^{2}}=
\frac{1}{a_{d,1}^{\chi_{r}^{2}}(\mathcal{D})\mathcal{L}(T)T^{d-\alpha}}
\left[\int_{\mathcal{D}(T)}F(\chi_{r}^{2}
(\mathbf{x}))d\mathbf{x}-C_{0}^{L}T^{d}|\mathcal{D}|\right],\label{eqfglr}
\end{equation}
\noindent where $a_{d,1}^{\chi_{r}^{2}}(\mathcal{D})$ is given in
(\ref{eqnck}) for $k=1.$ For $0<\alpha  <d/2$, under {\bfseries Conditions A2
and A3}, its   limit, in
distribution sense, $S_{\infty
}^{\chi_{r}^{2}},$ in the case of $F$ having Laguerre rank $k=1,$
 has characteristic function of the form
\begin{equation}
\phi (z)=\mathrm{E\exp }\{izS_{\infty }^{\chi_{r}^{2}}\}=\exp \left(
\frac{r}{2}\sum_{m=2}^{\infty }\frac{(-2iz/\sqrt{2r})^{m}
}{m}c_{m}\right),\quad z\in \mathbb{R}, \label{elcfl}
\end{equation}
\noindent where $c_{m},$ $m\geq 2,$ are defined as follows:
\begin{equation}
c_{m}=\int_{\mathcal{D}}\underset{(m)}{\cdots }\int_{\mathcal{D}}\frac{1}{\Vert \mathbf{x}_{1}-%
\mathbf{x}_{2}\Vert ^{\alpha }}\frac{1}{\Vert \mathbf{x}_{2}-\mathbf{x}%
_{3}\Vert ^{\alpha }}\cdots \frac{1}{\Vert \mathbf{x}_{m}-\mathbf{x}%
_{1}\Vert ^{\alpha }}d\mathbf{x}_{1}\dots d\mathbf{x}_{m}.
\label{eqcoefchfbcc}
\end{equation}

\end{theorem}
\begin{remark}
\label{redth}
Note that, from Theorem \ref{rep}, applied to the particular case of chi-squared random fields with $k=1,$
\begin{equation}S_{\infty
}^{\chi_{r}^{2}}=\lim_{T\longrightarrow
\infty}\frac{C_{1}^{L}}{a_{d,1}^{\chi_{r}^{2}}(\mathcal{D}
)\mathcal{L}(T)T^{d-\alpha }}\int_{\mathcal{D}
(T)}e_{1}^{(r/2)}(\chi_{r}^{2}
(\mathbf{x}))~d\mathbf{x}.\label{lchisquaredbb}
\end{equation}
\end{remark}

\begin{proof}
From  Remark \ref{redth} (see equation (\ref{idcorr}) and Theorem \ref{rep}), the limit
distribution of $S_{T}^{\chi_{r}^{2}}$ as $T\rightarrow \infty,$ if
it exists, can be obtained as the limit in distribution given in
 (\ref{lchisquaredbb}), since $F$ has Laguerre rank $k$ equal
to one.

The first Laguerre polynomial of the chi-square random field $\{
\chi_{r}^{2}(\mathbf{x}),\ \mathbf{x}\in \mathbb{R}^{d}\}$  is the
sum of $r$ independent copies of the second Hermite polynomial of
the original Gaussian random field $\{Y(\mathbf{x}),\ \mathbf{x}\in
\mathbb{R}^{d}\}$ involved, that is, for $\mathbf{x}\in
\mathbb{R}^{d},$ we have by (\ref{threepoly}),
\begin{equation}
e_{1}^{(r/2)}(\chi _{r}^{2}(\mathbf{x)})=\sqrt{\frac{2}{r}}\left(
\frac{r}{2}-\sum_{j=1}^{r}Y_{j}^{2}(\mathbf{x})\right)=
-\frac{1}{\sqrt{2r}}%
\sum_{j=1}^{r}(Y_{j}^{2}(\mathbf{x})-1)=-\frac{1}{\sqrt{2r}}%
\sum_{j=1}^{r}H_{2}(Y_{j}(\mathbf{x})).\label{eqflp}\end{equation}

From equation (\ref{eqflp}), one can  prove, in a   similar way to
Theorem 3.2 by Leonenko, Ruiz-Medina and Taqqu (2014), that  the
limit characteristic function admits the expansion  (\ref{elcfl}).
Specifically,
\begin{eqnarray} \phi_{T}(z)&=& E\left[\exp\left(
\frac{iz}{T^{d-\alpha
}\mathcal{L}(T)a_{d,1}^{\chi_{r}^{2}}(\mathcal{D})}\int_{\mathcal{D} (T)}\left(-\frac{1}{\sqrt{2r}}%
\sum_{j=1}^{r}H_{2}(Y_{j}(\mathbf{x}))\right)d\mathbf{x}\right)\right]\nonumber\\
&=& \prod_{j=1}^{r}\exp\left(\frac{1}{2} \sum_{m=2}^{\infty
}\frac{1}{m}\left(\frac{-2iz}{\sqrt{2r}a_{d,1}^{\chi_{r}^{2}}(\mathcal{D})T^{d-\alpha
}\mathcal{L}(T)}\right)^{m}\mbox{Tr}\left(
R_{Y,\mathcal{D} (T)}^{m}\right)\right)\nonumber\\
&=& \exp\left(\frac{r}{2} \sum_{m=2}^{\infty
}\frac{1}{m}\left(\frac{-2iz}{\sqrt{2r}a_{d,1}^{\chi_{r}^{2}}(\mathcal{D})T^{d-\alpha
}\mathcal{L}(T)}\right)^{m}\mbox{Tr}\left( R_{Y,\mathcal{D}
(T)}^{m}\right)\right).
 \label{chfseqttbb}
\end{eqnarray}
\noindent Note that under {\bfseries Condition A2}, since
$EY^{2}(\mathbf{x})=1,$
\begin{eqnarray}
\int_{\mathcal{D} (T)}d\mathbf{x}&=&\int_{\mathcal{D}
(T)}E\left[Y^{2}(\mathbf{x})\right]d\mathbf{x}=E\left[\int_{\mathcal{D}(T)}Y^{2}(\mathbf{x})d\mathbf{x}\right]=\sum_{j=1}^{\infty
}\lambda_{j,T}(R_{Y,\mathcal{D}(T)})E\eta_{j}^{2} \nonumber\\
&=&\sum_{j=1}^{\infty }\lambda_{j,T}(R_{Y,\mathcal{D}(T)}).\nonumber
\end{eqnarray}

In the study of the convergence of the series (\ref{chfseqttbb}), to  apply
Dominated Convergence Theorem, we use Theorem 3.1 by Leonenko,
Ruiz-Medina and Taqqu (2014), where it is proved that, for $0<\alpha
<d/2,$  the squared $\mathcal{K}_{\alpha }^{2}$ of the operator
\begin{equation}
\mathcal{K}_{\alpha }(f)(\mathbf{x})=\int_{\mathcal{D}}\frac{1}{\|\mathbf{x}-\mathbf{y%
}\|^{\alpha }}f(\mathbf{y})d\mathbf{y},\quad \forall f\in \mbox{Supp}(%
\mathcal{K}_{\alpha}),  \quad 0<\alpha <d, \label{RKD}
\end{equation}
\noindent is in the trace class.   In particular, its trace is given
by
\begin{equation} \mbox{Tr}\left( \mathcal{K}_{\alpha
}^{2}\right)=\int_{\mathcal{D}}\int_{\mathcal{D}}\frac{1}{\|\mathbf{x}-\mathbf{y}\|^{2\alpha
}}d\mathbf{x}d\mathbf{y}=[a_{d,1}^{\chi_{r}^{2}}(\mathcal{D})]^{2}<\infty.
\label{trk}
\end{equation}

\noindent From the Definition of the
Fredholm determinant of a trace operator (see, for example,  Simon, 2005, Chapter 5, pp.47-48, equation
(5.12)) the Fredholm determinant of $\mathcal{K}_{\alpha }^{2}$ is  given by
\begin{equation}
D_{\mathcal{K}_{\alpha }^{2}}(\omega )=\mbox{det}(I-\omega \mathcal{K}_{\alpha }^{2})=\exp\left(-\sum_{k=1}^{\infty }%
\frac{\mbox{Tr}\mathcal{K}_{\alpha
}^{2k}}{k}\omega^{k}\right)=\exp\left(-\sum_{k=1}^{\infty
}\sum_{l=1}^{\infty}[\lambda_{l}(\mathcal{K}_{\alpha
}^{2})]^{k}\frac{\omega^{k}}{k}\right), \label{fdfbb}
\end{equation}
\noindent for $\omega \in \mathbb{C},$ and $|\omega
|\|\mathcal{K}_{\alpha }^{2}\|_{1}< 1,$ with $\|\mathcal{K}_{\alpha }^{2}\|_{1}$ denoting the trace norm
of operator $\mathcal{K}_{\alpha }^{2}.$
In particular, for $\omega
=2\mathrm{i}z,$ and for $|z|< \frac{1}{2\|\mathcal{K}_{\alpha }^{2}\|_{1}},$
\begin{equation}
[D_{\mathcal{K}_{\alpha }^{2}}(2\mathrm{i}z )]^{-1/2}=
\exp\left(\frac{1}{2}\sum_{k=1}^{\infty }
\frac{\mbox{Tr}\mathcal{K}_{\alpha
}^{2k}}{k}(2\mathrm{i}z)^{k}\right)<\infty . \label{fdfbbb}
\end{equation}

In addition, under {\bfseries Condition A3}, there exists a positive
constant $C$ such that
\begin{eqnarray}
\frac{1}{d_{T}^{2}}\mbox{Tr}\left(
R_{Y,\mathcal{D}(T)}^{2}\right)&=&\int_{\mathcal{D}}\int_{\mathcal{D}}\frac{\mathcal{L}(T\|\mathbf{x}_{1}-\mathbf{x}_{2}\|)}{
\mathcal{L}(T)}\frac{\mathcal{L}(T\|\mathbf{x}_{2}-\mathbf{x}_{1}\|)}{\mathcal{L}(T)}
\frac{1}{\|\mathbf{x}_{1}-\mathbf{x}_{2}\|^{2\alpha
}}d\mathbf{x}_{1}d\mathbf{x}_{2}\nonumber\\
&\leq &
C\int_{\mathcal{D}}\int_{\mathcal{D}}\frac{1}{\|\mathbf{x}_{1}-\mathbf{x}_{2}\|^{2\alpha
}}d\mathbf{x}_{1}d\mathbf{x}_{2}=C \mbox{Tr}\left(
\mathcal{K}_{\alpha }^{2}\right)<\infty,
 \label{eqm2}
 \end{eqnarray}
 \begin{eqnarray}
& &\frac{1}{d_{T}^{m}}
 \mbox{Tr}\left( R_{Y,\mathcal{D}(T)}^{m}\right)=\nonumber\\ & & =\frac{1}{[\mathcal{L}(T)]^{m}}\int_{\mathcal{D}
}\underset{(m)}{\cdots }\int_{\mathcal{D}
}\frac{\mathcal{L}(T\|\mathbf{x}_{1}-\mathbf{x}_{2}\|)}{\|\mathbf{x}_{1}-\mathbf{x}_{2}\|^{\alpha
}}\frac{\mathcal{L}(T\|\mathbf{x}_{2}-\mathbf{x}_{3}\|)}{\|\mathbf{x}_{2}-\mathbf{x}_{3}\|^{\alpha
}}\cdots
\frac{\mathcal{L}(T\|\mathbf{x}_{m}-\mathbf{x}_{1}\|)}{\|\mathbf{x}_{m}-\mathbf{x}_{1}\|^{\alpha
}}d\mathbf{x}_{1}\dots
d\mathbf{x}_{m}\nonumber\\
& & \leq C   \int_{\mathcal{D}}\underset{(m)}{\cdots }\int_{\mathcal{D}
}\frac{1}{\|\mathbf{x}_{1}-\mathbf{x}_{2}\|^{\alpha
}}\frac{1}{\|\mathbf{x}_{2}-\mathbf{x}_{3}\|^{\alpha
}}\cdots \frac{1}{\|\mathbf{x}_{m}-\mathbf{x}_{1}\|^{\alpha }}d\mathbf{x}_{1}\dots d\mathbf{x}_{m} \nonumber\\
& & =C \mbox{Tr}\left( \mathcal{K}_{\alpha }^{m}
\right)<\infty,\quad  m>2,
 \label{chcoeff}
\end{eqnarray}
\noindent since $\| \mathcal{K}_{\alpha }^{m}\|_{1}\leq K\|
\mathcal{K}_{\alpha }^{2}\|_{1},$ $K>0,$ for $m>2.$ Here, $d_{T}=
a_{d,1}^{\chi_{r}^{2}}(\mathcal{D})T^{d-\alpha }\mathcal{L}(T).$

From equations (\ref{chfseqttbb}) and (\ref{fdfbbb})-(\ref{chcoeff}),
 for $0<z<\sqrt{r/2}\wedge \sqrt{r/\left(2\|\mathcal{K}_{\alpha
}^{2}\|_{1}^{2}\right)},$  we obtain

\begin{eqnarray}
 |\phi_{T}(z)|&\leq & \left|\exp\left(\frac{rC}{2} \sum_{m=2}^{\infty }\frac{1}{m}\left( -2\mathrm{i}z/\sqrt{2r}\right)^{m}\mbox{Tr}\left(
\mathcal{K}_{\alpha }^{m}\right)\right)\right|\nonumber\\
&=&\left|\exp\left(\frac{rC}{2} \left[\sum_{m=1}^{\infty
}\frac{1}{2m}\left( -2\mathrm{i}z/\sqrt{2r}\right)^{2m}\mbox{Tr}\left(
\mathcal{K}_{\alpha }^{2m}\right)\right.\right.\right.\nonumber\\
& &\left.\left.\left.
+\sum_{m=1}^{\infty
}\frac{1}{2m+1}\left(- 2\mathrm{i}z/\sqrt{2r}\right)^{2m+1}\mbox{Tr}\left(
\mathcal{K}_{\alpha
}^{2m+1}\right)\right]\right)\right|\nonumber\\
&\leq & \left|\exp\left(\frac{rC}{2} \left[\sum_{m=1}^{\infty
}\frac{1}{m}\left( -2\mathrm{i}z/\sqrt{2r}\right)^{m}\mbox{Tr}\left(
\mathcal{K}_{\alpha }^{2m}\right)\right.\right.\right.\nonumber\\
& &\left.\left.\left.
+\sum_{m=1}^{\infty
}\frac{1}{m}\left( -2\mathrm{i}z/\sqrt{2r}\right)^{m}\mbox{Tr}\left(
\mathcal{K}_{\alpha }^{2m}\right)\right]\right)\right|= \left|D_{\mathcal{K}_{\alpha
}^{2}}\left(\frac{-2\mathrm{i}z}{\sqrt{2r}} \right)
\right|^{-rC}<\infty ,
 \label{eqdct}
 \end{eqnarray}
 \noindent where the last identity in (\ref{eqdct}) is obtained from the definition of the Fredholm determinant of $\mathcal{K}_{\alpha
}^{2}$ as given in equation  (\ref{fdfbb}).

We can thus apply the Dominated Convergence Theorem to obtain
$\lim_{T\rightarrow \infty}\psi_{T}(z)=\psi (z),$ for
$0<z<\sqrt{r/2}\wedge \sqrt{r/\left(2\|\mathcal{K}_{\alpha
}^{2}\|_{1}^{2}\right)}.$
 An analytic
continuation argument (see Lukacs, 1970, Th. 7.1.1) guarantees that
$\psi $ defines the unique limit characteristic function for all
real values of $z.$

\end{proof}

\section{Limit theorems for Laguerre rank equal to  one and two and  Wiener-It$\hat{o}$ stochastic integral representations}
\label{secsr}
Consider the chi-squared field defined in (\ref{Chi_1}).
The multiple Wiener-It{\^o} stochastic integral representation of the limit
in distribution  of the functional (\ref{eqfglr}), and of the
functional
\begin{equation} S_{2,T}=\frac{1}{a_{d,2}^{\chi _{r}^{2}}(\mathcal{D}
)\mathcal{L}^{2}(T)T^{d-2\alpha }}\int_{\mathcal{D}
(T)}e_{2}^{(r/2)}(\chi _{r}^{2}(\mathbf{x}))~d\mathbf{x}
\label{functformula}
\end{equation}
\noindent is derived in Theorems  \ref{th1}  and \ref{par2} below,
respectively. Here, $a_{d,2}^{\chi
_{r}^{2}}$ is defined as in (\ref{eqnck}) for $k=2,$ and
$\mathcal{L}(T)$ is the slowly varying function introduced in
(\ref{cov}). The basis function $e^{r/2}_{2}(u)$ is defined in
Relation (\ref{2.2}).
 In the Section \ref{ISR}, from the multiple Wiener-It{\^o} stochastic integral representations
derived in this section, we obtain an infinite series representation, in
terms of independent  random variables,
for  $S_{\infty }^{\chi_{r}^{2}},$ with
characteristic function given in Theorem \ref{prchf}, and for
the random variable obtained as the limit in distribution of the
functional (\ref{functformula}) in Theorem \ref{par2}
below.

The slowly varying function $%
\mathcal{L}$ in (\ref{cov}) is assumed to belong to the class $\widetilde{\mathcal{L}}%
\mathcal{C}$ which is now introduced (see Definition 9 by Leonenko
and Olenko, 2013).

\begin{definition}
\label{defLC} An infinitely differentiable function $\mathcal{L}({\cdot})$
belongs to the class $\widetilde{\mathcal{L}}\mathcal{C}$ if

\begin{itemize}
\item[1.] for any $\delta >0,$ there exists $\lambda_{0}(\delta )>0$ such
that $\lambda^{-\delta }\mathcal{L}(\lambda )$ is decreasing and $%
\lambda^{\delta }\mathcal{L}(\lambda )$ is increasing if $\lambda
>\lambda_{0}(\delta );$

\item[2.] $\mathcal{L}_{j}\in \mathcal{S}\mathcal{L},$ for all $j\geq 0,$
where $\mathcal{L}_{0}(\lambda ):=\mathcal{L},$ $\mathcal{L}_{j+1}(\lambda
):=\lambda \mathcal{L}_{j}^{\prime }(\lambda ),$ with $\mathcal{S}\mathcal{L}
$ being the class of functions that are slowly varying at infinity and
bounded on each finite interval.
\end{itemize}
\end{definition}

The following lemma will be applied in the proofs of  Theorem
\ref{th1}  and   \ref{par2} below (see Theorem 11 by
Leonenko and Olenko, 2013).

\begin{lemma}
\label{LeonenkoOlenko12} Let $\alpha \in (0,d),$ $S\in C^{\infty
}(s_{n-1}(1)),$ and $\mathcal{L}\in
\widetilde{\mathcal{L}}\mathcal{C}.$ Let $\{X(\mathbf{x}),\
\mathbf{x}\in \mathbb{R}^{d}\}$ be a mean-square continuous
homogeneous random field with zero mean. Let the field $X$ has
spectral density $f(\mathbf{u}),$ $\mathbf{u}\in \mathbb{R}^{d},$
which is infinitely differentiable for all $\mathbf{u}\neq 0.$ If
the covariance function $B(\mathbf{x}),$ $\mathbf{x}\in
\mathbb{R}^{d},$ of the field $X$ has the following behavior

\begin{itemize}
\item[(a)] $\|\mathbf{x}\|^{\alpha }B(\mathbf{x})\sim S\left( \frac{\mathbf{x%
}}{\|\mathbf{x}\|}\right)\mathcal{L}(\|\mathbf{x}\|),\quad \mathbf{x}%
\longrightarrow \infty ,$

\noindent the spectral density satisfies the condition

\item[(b)] $\|\mathbf{u}\|^{d-\alpha }f(\mathbf{u})\sim \widetilde{S}%
_{\alpha ,d}\left( \frac{\mathbf{u}}{\|\mathbf{u}\|}\right)\mathcal{L}\left(%
\frac{1}{\|\mathbf{u}\|}\right),\quad \| \mathbf{u}\|\longrightarrow 0.$
\end{itemize}
\end{lemma}
In Propositions \ref{fprop}--\ref{par} and  Theorems \ref{th1}--\ref{par2} below, the following Fourier transforms and convolution
formulae will be applied in  $\mathcal{S}(\mathbb{R}^{d}),$ the space of
  infinitely differentiable functions on
$\mathbb{R}^{d},$  whose derivatives remain bounded when
multiplied by polynomials, i.e., whose derivatives are rapidly
decreasing
(see Lemma 1 of Stein, 1970, p.117).
\begin{lemma}
\label{l1s117}
\begin{itemize}
\item[(i)] The Fourier transform of the function
$\|\mathbf{z}\|^{-d+\beta}$ is $\nu (\beta )\|\mathbf{z}\|^{-\beta },$ in the sense that
\begin{equation}
\int_{\mathbb{R}^{d}}\|\mathbf{z}\|^{-d+\beta }\overline{\psi(\mathbf{z})}d\mathbf{z}= \int_{\mathbb{R}^{d}}\nu
(\beta )\|\mathbf{z}\|^{-\beta }\overline{\mathcal{F}(\psi )(\mathbf{z})}d\mathbf{z},\quad \forall \psi \in
\mathcal{S}(\mathbb{R}^{d}), \label{fist117}
\end{equation}
\noindent where
\begin{equation}\nu (\beta
)=\frac{\pi^{d/2}2^{\beta }\Gamma (\beta/2)}{\Gamma
\left(\frac{d-\beta}{2}\right)},\quad 0<\beta <d,\label{eRc}
\end{equation}
\noindent and
$$\mathcal{F}(\psi
)(\mathbf{z})=\int_{\mathbb{R}^{d}}\exp\left(-\mathrm{i}\left\langle
\mathbf{x},\mathbf{z}\right\rangle \right)\psi
(\mathbf{x})d\mathbf{x}$$ \noindent denotes the Fourier transform of
$\psi.$
\item[(ii)] The identity $\mathcal{F}\left((-\Delta
)^{-\beta /2}(f)\right)(\mathbf{z})=\|\mathbf{z}\|^{-\beta }\mathcal{F}(f)(\mathbf{z})$ holds in the sense that
\begin{equation}
\int_{\mathbb{R}^{d}} (-\Delta )^{-\beta
/2}(f)(\mathbf{x})\overline{g(\mathbf{x})}d\mathbf{x}=\frac{1}{(2\pi)^{d}}\int_{\mathbb{R}^{d}}\mathcal{F}(f)(\mathbf{x})\|\mathbf{x}\|^{-\beta
}\overline{\mathcal{F}(g)(\mathbf{x})}d\mathbf{x},\quad \forall
f,g\in \mathcal{S}(\mathbb{R}^{d}), \label{irpp}
\end{equation}
\noindent for  $0<\beta <d.$
\item[(iii)] The following convolution formula is   obtained by
iteration of (\ref{irpp})

\begin{eqnarray}
& & \int_{\mathbb{R}^{d}}\frac{1}{\nu (4\beta)}\|\mathbf{z}\|^{-d+4\beta }\overline{f(\mathbf{z})}d\mathbf{z}=
\int_{\mathbb{R}^{d}}\|\mathbf{z}\|^{-4\beta }\overline{\mathcal{F}(f
)(\mathbf{z})}d\mathbf{z}\nonumber\\
& &=\int_{\mathbb{R}^{d}}\frac{1}{[\nu (\beta
)]^{4}}\left[\int_{\mathbb{R}^{3d}}\|\mathbf{z}-\mathbf{x}_{1}\|^{-d+\beta
}\|\mathbf{x}_{1}-\mathbf{x}_{2}\|^{-d+\beta
}\|\mathbf{x}_{2}-\mathbf{y}\|^{-d+\beta }\|\mathbf{y}\|^{-d+\beta
}d\mathbf{x}_{1}d\mathbf{x}_{2}d\mathbf{y}\right]\nonumber\\
& &\hspace*{3cm}\times
\overline{f(\mathbf{z})}d\mathbf{z},\quad \forall
f\in \mathcal{S}(\mathbb{R}^{d}),\quad
0<\beta<d/4.\nonumber\\\label{cfsf2}
\end{eqnarray}
\end{itemize}
\end{lemma}

The proof of this lemma can be seen in Stein (1970, p.117), and Leonenko, Ruiz-Medina and Taqqu (2014).
\begin{proposition}
\label{fprop}
 For $0<\alpha <d/2,$ the following identities hold:
\begin{eqnarray}
\int_{\mathbb{R}^{2d}}\left\vert K\left( \boldsymbol{\lambda }_{1}+%
\boldsymbol{\lambda }_{2},\mathcal{D}\right) \right\vert ^{2}\frac{d\boldsymbol{%
\lambda }_{1}d\boldsymbol{\lambda }_{2}}{\left( \left\Vert \boldsymbol{%
\lambda }_{1}\right\Vert \left\Vert \boldsymbol{\lambda
}_{2}\right\Vert \right) ^{d-\alpha
}}&=&\left[\frac{a_{d,1}^{\chi_{r}^{2}}(\mathcal{D})\nu(\alpha )
}{|\mathcal{D}|}\right]^{2} = \frac{[\nu(\alpha
)]^{2}Tr(\mathcal{K}_{\alpha
}^{2})}{|\mathcal{D}|^{2}}<\infty ,\nonumber\\
\label{6.7.1}
\end{eqnarray}%
\noindent where   $a_{d,1}^{\chi_{r}^{2}}(\mathcal{D})$ is defined
in (\ref{eqnck}), $\nu (\alpha )$ is introduced in equation
(\ref{eRc}),
 and $K$ is the characteristic function
of the uniform distribution over set $\mathcal{D},$ given by
\begin{equation}
K\left( \boldsymbol{\lambda },\mathcal{D}\right)
=\int_{\mathcal{D}}e^{-\mathrm{i}\left\langle \boldsymbol{\lambda
},\mathbf{x}\right\rangle }p_{\mathcal{D}}\left( \mathbf{x}\right)
d\mathbf{x}=\frac{1}{\left\vert \mathcal{D}\right\vert }\int_{\mathcal{D}}e^{-\mathrm{i}%
\left\langle \boldsymbol{\lambda },\mathbf{x}\right\rangle }d\mathbf{x}=%
\frac{\vartheta (\boldsymbol{\lambda })}{\left\vert \mathcal{D}\right\vert },
\label{UD1}
\end{equation}%
\noindent with associated probability density function $p_{\mathcal{D}}\left( \mathbf{x%
}\right) =1/\left\vert \mathcal{D}\right\vert $ if $\mathbf{x}\in \mathcal{D},$ and $0$
otherwise.
\end{proposition}

\begin{remark}
Note that for  $\mathcal{D}=\mathcal{B}(%
\mathbf{0})=\{ \mathbf{x}\in \mathbb{R}^{d};\ \|\mathbf{x}\| \leq 1\},$ the
function $\vartheta (\boldsymbol{\lambda })$ in (\ref{UD1}) is of the form:
\begin{equation*}
\int_{\mathcal{B}(\mathbf{0})} \exp\left(\mathrm{i}\left\langle \mathbf{x},%
\boldsymbol{\lambda }\right\rangle\right)d\mathbf{x}=(2\pi )^{d/2} \frac{%
\mathcal{J}_{d/2}\left( \|\boldsymbol{\lambda }\|\right)}{\|\boldsymbol{%
\lambda }\|^{d/2}},\quad d\geq 2,
\end{equation*}
\noindent where $\mathcal{J}_{\nu }(\mathbf{z})$ is the Bessel function of
the first kind and order $\nu > -1/2.$ For a rectangle, $\mathcal{D}=\prod =\left\{
a_{i}\leq x_{i}\leq b_{i},\ i=1,\dots,d\right\},\quad \mathbf{0}\in \prod,$
\begin{equation*}
\vartheta (\boldsymbol{\lambda })= \prod_{j=1}^{d}\left( \exp\left( \mathrm{i}\lambda_{j}b_{j}\right)-\exp\left(
\mathrm{i}\lambda_{j}a_{j}\right)\right)/\mathrm{i}\lambda_{j},\quad d\geq 1.
\end{equation*}
\noindent (see, for example, Leonenko and Olenko, 2014).
\end{remark}

\begin{theorem}
\label{th1}
Assume that {\bfseries Conditions A2-A3} hold, $0<\alpha <d/2,$ and that $%
\mathcal{L}\in \widetilde{\mathcal{L}}\mathcal{C}.$ Consider
$S_{T}^{\chi_{r}^{2}}$ be the functional (\ref{eqfglr}), given in terms of the integral of functional $F$ of the chi-squared random field with
Laguerre rank equal to one. As
$T\rightarrow \infty,$ the limiting distribution
$S_{\infty}^{\chi_{r}^{2}}$ of $S_{T}^{\chi_{r}^{2}},$
 with characteristic function
(\ref{elcfl}), admits the following double Wiener-It{\^o} stochastic
integral representation:

\begin{eqnarray}
& & S_{\infty }^{\chi_{r}^{2}}= - \frac{|\mathcal{D}|}{\nu (\alpha
)\sqrt{2r}}\sum_{j=1}^{r}\int_{\mathbb{R}^{2d}}^{\prime
 } H(\boldsymbol{\lambda }_{1},\boldsymbol{\lambda }_{2}) \frac{%
Z_{j}\left( d\boldsymbol{\lambda }_{1}\right)Z_{j}\left(
d\boldsymbol{\lambda }_{2}\right)}{\left\Vert \boldsymbol{\lambda }_{1}\right\Vert ^{\frac{%
d-\alpha }{2}}\left\Vert \boldsymbol{\lambda }_{2}\right\Vert ^{\frac{%
d-\alpha }{2}}} \label{sirlr1}
\end{eqnarray}

\noindent where $Z_{j},$ $j=1,\dots,r,$ are independent  Gaussian
white noise measures, $\nu $ is defined in (\ref{eRc}), and the notation
$\int_{\mathbb{R}^{2d}}^{\prime }$ means that one does not
integrate
on the hyperdiagonals $\boldsymbol{\lambda }_{1}=\pm \boldsymbol{\lambda }%
_{2}.$ Here,
\begin{equation}
H\left( \boldsymbol{\lambda }_{1} ,\boldsymbol{\lambda
}_{2}\right)=K\left( \boldsymbol{\lambda } _{1}+\boldsymbol{\lambda
} _{2},\mathcal{D} \right),  \label{eqH}
\end{equation}
\noindent where $K\left( \boldsymbol{\lambda },\mathcal{D} \right)$ is defined in (\ref{UD1}).

\end{theorem}

The proofs  of Proposition \ref{fprop}  and Theorem \ref{th1} can be derived as in Theorem 4.1 by Leonenko, Ruiz-Medina
and Taqqu (2014), since from Theorems \ref{rep} and
\ref{prchf}, Theorem \ref{th1} holds for the functional
$S_{1,T}$ given  by (see (\ref{lchisquaredbb}))
\begin{eqnarray}
& & S_{1,T}=\frac{1}{a_{d,1}^{\chi_{r}^{2}}(\mathcal{D}
)\mathcal{L}(T)T^{d-\alpha }}\int_{\mathcal{D}
(T)}e_{1}^{(r/2)}(\chi
_{r}^{2}(\mathbf{x}))~d\mathbf{x}\nonumber\\
& & = -\left[\frac{1}{a_{d,1}^{\chi_{r}^{2}}(\mathcal{D}
)\mathcal{L}(T)T^{d-\alpha
}}\right]\left[\frac{1}{\sqrt{2r}}\sum_{j=1}^{r}\int_{\mathcal{D}
(T)} H_{2}(Y_{j}(\mathbf{x}))d\mathbf{x}\right].
\label{functformulaL1}
\end{eqnarray}
We now turn to the case $k=2.$
\begin{proposition}
\label{par} Let $\mathcal{D}$ be a regular compact set and let $K\left( \boldsymbol{\lambda },\mathcal{D} \right)$ be defined in (\ref{UD1}).

 For $0<\alpha <d/4,$ the following identities hold:
\begin{equation}
 \int_{\mathbb{R}^{4d}}\left\vert K\left( \boldsymbol{\lambda }_{1}+%
\boldsymbol{\lambda }_{2}+\boldsymbol{\lambda }_{3}+\boldsymbol{\lambda }_{4},\mathcal{D}\right) \right\vert ^{2}\frac{\prod_{i=1}^{4}d\boldsymbol{%
\lambda }_{i}}{\prod_{i=1}^{4} \left(\left\Vert \boldsymbol{%
\lambda }_{i}\right\Vert  \right) ^{d-\alpha
}}
=\frac{[a_{d,2}^{\chi_{r}^{2}}(\mathcal{D})]^{2}[\nu(\alpha )]^{4}}{|\mathcal{D}|^{2}}<\infty ,
\label{6.7.1b}
\end{equation}%
\noindent where   $a_{d,2}^{\chi_{r}^{2}}(\mathcal{D})$ is defined as in equation
(\ref{eqnck}) for $k=2,$ and $\nu (\alpha )$ is introduced in
equation (\ref{eRc}).
\end{proposition}
\begin{proof}
The proof  follows from the application of Theorem 3.1
in Leonenko, Ruiz-Medina and Taqqu (2014), where the asymptotic spectral properties of  operator
 $\mathcal{K}_{\alpha }$  in equation (\ref{RKD}), on a Dirichlet regular compact domain $\mathcal{D},$   are established.
Let us now consider the following norm  on $\mathcal{S}(\mathbb{R}^{d}),$

\begin{eqnarray} &&\|f\|^{2}_{(-\Delta
)^{2\alpha -d/2}}
=\left\langle (-\Delta )^{2\alpha -d/2}(f),f\right\rangle_{L^{2}(\mathbb{R}^{d})}\nonumber\\
& &
= \int_{\mathbb{R}^{d}}(-\Delta )^{2\alpha
-d/2}(f)(\mathbf{x})\overline{f(\mathbf{x})}d\mathbf{x}
=\int_{\mathbb{R}^{d}}\frac{1}{\nu (d-4\alpha
)}\int_{\mathbb{R}^{d}} \frac{1}{\|\mathbf{x}-\mathbf{y}\|^{4\alpha
}}f(\mathbf{y})\overline{f(\mathbf{x})}d\mathbf{y}d\mathbf{x}\nonumber\\
 && =\frac{1}{(2\pi)^{d}}\int_{\mathbb{R}^{d}} |\mathcal{F}(f)(\boldsymbol{\lambda
})|^{2}\|\boldsymbol{\lambda }\|^{-(d-4\alpha ) }d\boldsymbol{\lambda },\quad \forall f\in
\mathcal{S}(\mathbb{R}^{d}),\quad 0<\alpha <d/4. \label{eqrkdef}
\end{eqnarray}
 \noindent The  space $\mathcal{H}_{4\alpha -d}= \overline{\mathcal{S}(\mathbb{R}^{d})}^{\|\cdot \|_{(-\Delta
)^{2\alpha -d/2}}}$ is the Hilbert space
 of the functions of $\mathcal{S}(\mathbb{R}^{d})$
with the inner product
\begin{equation}\left\langle f,g\right\rangle_{(-\Delta )^{2\alpha
-d/2}} =
\int_{\mathbb{R}^{d}}\frac{1}{\nu (d-4\alpha
)}\int_{\mathbb{R}^{d}} \frac{1}{\|\mathbf{x}-\mathbf{y}\|^{4\alpha
}}f(\mathbf{y})\overline{g(\mathbf{x})}d\mathbf{y}d\mathbf{x},\quad
\forall f,g \in \mathcal{S}(\mathbb{R}^{d}),\label{ip}
\end{equation}
\noindent and the associated norm (\ref{eqrkdef}). Here, $\overline{\mathcal{S}(\mathbb{R}^{d})}^{\|\cdot \|_{(-\Delta
)^{2\alpha -d/2}}}$ denotes the closure of $\mathcal{S}(\mathbb{R}^{d})$ with the
norm (\ref{eqrkdef}).
Note that Equations (\ref{eqrkdef}) and (\ref{ip}) can be extended to the space $\mathcal{H}_{4\alpha -d}$ by
continuity of the norm. In particular,

\begin{equation}
\|1_{\mathcal{D}}\|_{\mathcal{H}_{4\alpha-d}}^{2}=
\int_{\mathcal{D}}\frac{1}{\nu (d-4\alpha
)}\int_{\mathcal{D}}\frac{1}{\|\mathbf{x}-\mathbf{y}\|^{4\alpha
}}d\mathbf{y}d\mathbf{x} =
\frac{[a_{d,2}^{\chi_{r}^{2}}(\mathcal{D})]^{2}}{\nu (d-4\alpha )}.
\label{TRkalpha}
\end{equation}

As noted before, from  Theorem 3.1 by Leonenko, Ruiz-Medina and Taqqu (2014),
\begin{equation}\mathrm{Tr}(\mathcal{K}_{\alpha }^{2})=\int_{\mathcal{D}}\int_{\mathcal{D}}\frac{1}{\|\mathbf{x}-\mathbf{y}\|^{2\alpha
}}d\mathbf{y}d\mathbf{x}< \infty,\quad 0<\alpha <d/2.\label{eqtrace}
\end{equation}
\noindent Thus, for   $\alpha = 2\beta ,$
$$\int_{\mathcal{D}}\int_{\mathcal{D}}\frac{1}{\|\mathbf{x}-\mathbf{y}\|^{4\beta
}}d\mathbf{y}d\mathbf{x}< \infty,\quad 0<\beta <d/4.$$

\noindent Therefore,
$$[a_{d,2}^{\chi_{r}^{2}}(\mathcal{D})]^{2}=\int_{\mathcal{D}}\int_{\mathcal{D}}\frac{1}{\|\mathbf{x}-\mathbf{y}\|^{4\alpha
}}d\mathbf{y}d\mathbf{x}=\nu (d-4\alpha )\|1_{\mathcal{D}}\|_{\mathcal{H}_{4\alpha-d}}^{2}<\infty,\quad 0<\alpha <d/4.$$
\noindent Equivalently, $1_{\mathcal{D}}$ belongs to the Hilbert space
$\mathcal{H}_{4\alpha -d},$ for $0<\alpha <d/4.$

Applying the convolution formula  (\ref{cfsf2}) in  Lemma  \ref{l1s117}, we then obtain
\begin{eqnarray}
\frac{[a_{d,2}^{\chi_{r}^{2}}(\mathcal{D})]^{2}}{\nu (d-4\alpha )}&=&\|1_{\mathcal{D}}\|_{\mathcal{H}_{4\alpha -d}}^{2}= \frac{|\mathcal{D}|^{2}}{(2\pi
)^{d}}\int_{\mathbb{R}^{d}}|K(\boldsymbol{\omega}_{1},\mathcal{D})|^{2}
\|\boldsymbol{\omega}_{1}\|^{-d+4\alpha }d\boldsymbol{\omega}_{1}\nonumber\\
& &=\frac{|\mathcal{D}|^{2}}{(2\pi )^{d}}\frac{\nu (4\alpha )}{[\nu (\alpha
)]^{4}}\int_{\mathbb{R}^{d}}|K(\boldsymbol{\omega}_{1},\mathcal{D})|^{2}
\left[\int_{\mathbb{R}^{3d}}\|\boldsymbol{\omega}_{1}-\boldsymbol{\omega}_{2}\|^{-d+\alpha
}\|\boldsymbol{\omega}_{2}-\boldsymbol{\omega}_{3}\|^{-d+\alpha
}\right.
\nonumber\\
& & \hspace*{1cm}\left.\times \|\boldsymbol{\omega}_{3}-\boldsymbol{\omega}_{4}\|^{-d+\alpha
}\|\boldsymbol{\omega}_{4}\|^{-d+\alpha
}\prod_{i=2}^{4}d\boldsymbol{\omega}_{i}\right]d\boldsymbol{\omega}_{1}\nonumber\\
 &=&\frac{|\mathcal{D}|^{2}\nu (4\alpha )}{(2\pi
)^{d}[\nu(\alpha )]^{4}}\int_{\mathbb{R}^{4d}}\left|K\left(\sum_{i=1}^{4}\boldsymbol{\lambda}_{i},\mathcal{D}\right)\right|^{2}
\frac{\prod_{i=1}^{4}d\boldsymbol{\lambda}_{i}}{\prod_{i=1}^{4}\|\boldsymbol{\lambda}_{i}\|^{d-\alpha
}}\nonumber.
  \label{eqsiv}
\end{eqnarray}
Hence, $$[a_{d,2}^{\chi_{r}^{2}}(\mathcal{D})]^{2} =\frac{|\mathcal{D}|^{2}}{[\nu (\alpha
)]^{4}}\int_{\mathbb{R}^{4d}}\left|K\left(\sum_{i=1}^{4}\boldsymbol{\lambda}_{i},\mathcal{D}\right)\right|^{2}
\frac{\prod_{i=1}^{4}d\boldsymbol{\lambda}_{i}}{\prod_{i=1}^{4}\|\boldsymbol{\lambda}_{i}\|^{d-\alpha
}},$$

\noindent  since $ \frac{\nu (4\alpha )\nu (d-4\alpha
)}{(2\pi )^{d}}=1.$  Equation (\ref{6.7.1b}) then holds.

\noindent Note that by continuity of the norm in $\mathcal{H}_{4\alpha -d}$
$$1_{\mathcal{D}}\star
1_{\mathcal{D}}(\mathbf{x})=\int_{\mathbb{R}^{d}}1_{\mathcal{D}}(\mathbf{y})1_{\mathcal{D}}(\mathbf{x}+\mathbf{y})d\mathbf{y}=
\int_{\mathcal{D}}1_{\mathcal{D}}(\mathbf{x}+\mathbf{y})d\mathbf{y}
\in L^{2}(\mathcal{D})\subseteq \mathcal{H}_{4\alpha -d},$$ \noindent since
$$\int_{\mathbb{R}^{d}}\left|\int_{\mathbb{R}^{d}}1_{\mathcal{D}}(\mathbf{y})1_{\mathcal{D}}(\mathbf{x}+\mathbf{y})d\mathbf{y}\right|^{2}d\mathbf{x}\leq
\left|\mathcal{B}_{R(\mathcal{D})}(\mathbf{0})\right|^{3},$$
\noindent where $|\mathcal{B}_{R(\mathcal{D})}(\mathbf{0})|$ denotes the
Lebesgue measure of the ball of center $\mathbf{0}$ and radius
$R(\mathcal{D}),$ with $R(\mathcal{D})$ being equal to two times the diameter of the
regular compact set $\mathcal{D}$ containing the point $\mathbf{0}.$ Hence,
$\mathcal{F}(1_{\mathcal{D}}\star 1_{\mathcal{D}})(\boldsymbol{\lambda
})=|\mathcal{D}|^{2}|K(\boldsymbol{\lambda },\mathcal{D})|^{2}$ belongs to the space of
Fourier transforms of functions in $\mathcal{H}_{4\alpha -d}$ (see also Remark 3.1 by Leonenko, Ruiz-Medina and Taqqu, 2014).

\end{proof}

\medskip

Theorem \ref{prchf} provided the limit of (\ref{chisqrfrg1}) involving $e_{1}^{r/2}.$ The next theorem provides the limit of (\ref{functformula}),
involving $e_{2}^{r/2}.$ Note that $e_{2}^{r/2}$ is defined in (\ref{threepoly}), but also satisfies (\ref{lagHermf}) below.
\begin{theorem}
\label{par2}
Assume that {\bfseries Conditions A2-A3} hold, and that $%
\mathcal{L}\in \widetilde{\mathcal{L}}\mathcal{C}.$ Then,  for \linebreak $0<\alpha <d/4,$ the
functional $S_{2,T}$ defined in (\ref{functformula}) converges in distribution to
 the random
variable $S_{\infty}$ admitting  the following multiple Wiener-It{\^o} stochastic
integral representation:
\begin{eqnarray}
& & S_{\infty }\underset{d}{=} \frac{|\mathcal{D}|}{4[\nu (\alpha
)]^{2}}\left[r\left(\frac{r}{2}+1\right)\right]^{-1/2} \nonumber\\ &
& \hspace*{1cm}\times \left[\sum_{k,j; k\neq
j}^{r}\int_{\mathbb{R}^{2d}}^{\prime }\int_{\mathbb{R}^{2d}}^{\prime
}K\left( \sum_{i=1}^{4}\boldsymbol{\lambda}_{i} ,\mathcal{D}\right)
\frac{Z_{j}(d\boldsymbol{\lambda }_{1})Z_{j}(d\boldsymbol{\lambda
}_{2})Z_{k}(d\boldsymbol{\lambda }_{3})Z_{k}(d\boldsymbol{\lambda
}_{4})}{\prod_{i=1}^{4}\|\boldsymbol{\lambda }_{i}\|^{(d-\alpha
)/2}}\right.\nonumber\\ & &
\left.\hspace*{2cm}-\sum_{k=1}^{r}\int_{\mathbb{R}^{4d}}^{\prime
\prime }K\left( \sum_{i=1}^{4 }\boldsymbol{\lambda}_{i}
,\mathcal{D}\right)\frac{\prod_{i=1}^{4}Z_{k}(d\boldsymbol{\lambda
}_{i})}{\prod_{i=1}^{4}\|\boldsymbol{\lambda }_{i}\|^{(d-\alpha
)/2}}\right],\label{eqres1}
\end{eqnarray}
\noindent  where the random measures $Z_{j}(\cdot ),$ $j=1,2,3,4,$
are independent Wiener measures. $K(\cdot,\mathcal{D})$ is the
characteristic function of the uniform distribution over the set
$\mathcal{D}.$ The stochastic integrals
$\int_{\mathbb{R}^{2d}}^{\prime }$ appearing in the first sum of
(\ref{eqres1}) are defined as mean square integrals, in which
integration is excluded over hyperdiagonals  $\boldsymbol{\lambda
}_{1}=\pm \boldsymbol{\lambda }_{2},$
 and  $\boldsymbol{\lambda }_{3}=\pm \boldsymbol{\lambda }_{4},$
    related to each component $Z_j$ and $Z_k$ (see Fox and Taqqu, 1985).
In the second sum,
 $\int_{\mathbb{R}^{4d}}^{\prime \prime }$ means that one can not integrate
on the hyperdiagonals $\boldsymbol{\lambda }_{i}=\pm \boldsymbol{\lambda }_{j},$ $i\neq j,$ $i,j=1,2,3,4.$

\end{theorem}

\begin{proof}

 The restriction to $\mathcal{D}(T)$ of the
independent copies $Y_{j},$ $j=1,\dots,r,$ of Gaussian random field
$Y,$ i.e., $\{Y_{j}(\mathbf{x}),\ \mathbf{x}\in \mathcal{D}(T),\
j=1,\dots,r\},$ satisfying {\bfseries Conditions A2--A3}, admit the
following stochastic integral representation:
\begin{equation}Y_{j}(\mathbf{x})=\frac{|\mathcal{D} (T)|}{(2\pi)^{d}}\int_{\mathbb{R}^{d}}\exp\left(\mathrm{i}\left\langle \mathbf{x}, \boldsymbol{\lambda }\right\rangle\right)
K\left( \boldsymbol{\lambda },\mathcal{D}(T)\right)f_{0}^{1/2}(\boldsymbol{\lambda })Z_{j}(d\boldsymbol{\lambda }),\quad
\mathbf{x}\in \mathcal{D}(T),\quad j=1,\dots,r.
\label{eqgrr}
\end{equation}

It is well-known (see, for example, Anh and Leonenko, 1999) that
\begin{eqnarray}
e_{2}^{(r/2)}(\chi _{r}^{2}(\mathbf{x})) &=& \frac{1}{4}\left(r\left(\frac{r}{2}+1\right)\right)^{-1/2}\nonumber\\
&\times & \left[\sum_{k,j=1,\ k\neq j}^{r}H_{2}(Y_{k}(\mathbf{x}))H_{2}(Y_{j}(\mathbf{x}))-\sum_{k=1}^{r}H_{4}(Y_{k}(\mathbf{x}))\right],\nonumber\\
\label{lagHermf}
\label{eqident}
\end{eqnarray}
\noindent where, as before, $\chi _{r}^{2}(\mathbf{x})$
 is the chi-squared random field introduced in (\ref{Chi_1}), and
 $e_{2}^{(r/2)}$ denotes the
second Laguerre polynomial with index $r/2$ (see Bateman and
Erdelyi, 1953, Chapter 10). Here, $H_{2}(u)=u^{2}-1$ is the second
Chebyshev-Hermite polynomial, and $H_{4}(u)=u^{4}-6u^{2}+3$ is the
fourth Chebyshev-Hermite polynomial.

From equation (\ref{lagHermf}), the functional (\ref{functformula})
admits the following  representation:
\begin{eqnarray}
S_{2,T}&=&\frac{1}{a_{d,2}^{\chi _{r}^{2}}(\mathcal{D}
)\mathcal{L}^{2}(T)T^{d-2\alpha }}\int_{\mathcal{D}
(T)}e_{2}^{r/2}(\chi _{r}^{2}(\mathbf{x}))~d\mathbf{x}\nonumber\\
&=& \frac{1}{4}\left(r\left(\frac{r}{2}+1\right)\right)^{-1/2}\frac{1}{d_{T}}\left[\sum_{k,j=1,\ k\neq j}^{r}\int_{\mathcal{D}(T)}H_{2}(Y_{k}(\mathbf{x}))H_{2}(Y_{j}(\mathbf{x}))d\mathbf{x}\right.\nonumber\\
&&\left.-\sum_{k=1}^{r}\int_{\mathcal{D}(T)}H_{4}(Y_{k}(\mathbf{x}))d\mathbf{x}\right].\label{hermiterepfunctL}
\end{eqnarray}

Using It$\hat{o}$'s formula (see,
for example, Dobrushin and Major, 1979; Major, 1981), we obtain from equation (\ref{hermiterepfunctL})
\begin{eqnarray}
S_{2,T}&=&\frac{1}{4}\left(r\left(\frac{r}{2}+1\right)\right)^{-1/2}\frac{1}{d_{T}}\left[\sum_{k,j=1,\
k\neq j}^{r}\int_{\mathcal{D}(T)} \int_{\mathbb{R}^{2d}}^{\prime
}\int_{\mathbb{R}^{2d}}^{\prime
}\exp\left(\left\langle
\mathbf{x},\sum_{i=1}^{4}\boldsymbol{\lambda}_{i}\right\rangle
\right)\right.
\nonumber\\
& &
\left.\hspace*{4cm}\times \prod_{i=1}^{4}\sqrt{f_{0}(\|\boldsymbol{\lambda}_{i}\|)}\prod_{i=1}^{2}Z_{j}(d\boldsymbol{\lambda }_{i})\prod_{i=3}^{4}Z_{k}(d\boldsymbol{\lambda }_{i})\right.
\nonumber\\
& &\left. -\sum_{k=1}^{r}\int_{\mathcal{D}(T)}
\int_{\mathbb{R}^{4d}}^{\prime \prime}\exp\left(\left\langle
\mathbf{x},\sum_{i=1}^{4}\boldsymbol{\lambda}_{i}\right\rangle
\right)
\prod_{i=1}^{4}\sqrt{f_{0}(\|\boldsymbol{\lambda}_{i}\|)}Z_{k}(d\boldsymbol{\lambda }_{i})\right]\nonumber\\
&=&\frac{1}{4}\left(r\left(\frac{r}{2}+1\right)\right)^{-1/2}\frac{|\mathcal{D}|}{[\nu(\alpha
)]^{2}d_{T}}\left[\sum_{k,j=1,\ k\neq j}^{r}
\int_{\mathbb{R}^{2d}}^{\prime
}\int_{\mathbb{R}^{2d}}^{\prime
}K\left(\sum_{i=1}^{4}\boldsymbol{\lambda}_{i},\mathcal{D}\right)\right.
\nonumber\\
& &
\left.\hspace*{3.5cm}\times \left([\nu (\alpha
)]^{2}\prod_{i=1}^{4}\sqrt{f_{0}(\|\boldsymbol{\lambda}_{i}\|/T)}\right)\prod_{i=1}^{2}Z_{j}(d\boldsymbol{\lambda }_{i})\prod_{i=3}^{4}Z_{k}(d\boldsymbol{\lambda }_{i})\right.
\nonumber\\
& &\left. -\sum_{k=1}^{r}\int_{\mathbb{R}^{4d}}^{\prime
\prime}K\left(\sum_{i=1}^{4}\boldsymbol{\lambda}_{i},\mathcal{D}\right)
\left([\nu (\alpha )]^{2}\prod_{i=1}^{4}\sqrt{f_{0}(\|\boldsymbol{\lambda}_{i}\|/T)}\right)Z_{k}(d\boldsymbol{\lambda }_{i})\right].\nonumber\\
\label{hermiterepfunctL2}
\end{eqnarray}

Hence, applying Minkowski inequality,

\begin{eqnarray}& & E\left[S_{2,T}-\frac{|\mathcal{D}|}{4[\nu (\alpha )]^{2}}\left[r\left(\frac{r}{2}+1\right)\right]^{-1/2}\left[\sum_{k,j;
k\neq j}^{r}\int_{\mathbb{R}^{2d}}^{\prime }\int_{\mathbb{R}^{2d}}^{\prime }K\left(
\sum_{i=1}^{4 }\boldsymbol{\lambda}_{i}
,\mathcal{D}\right)\right.\right.\nonumber\\
& &\left.\left.\hspace*{1cm}\times
\frac{Z_{j}(d\boldsymbol{\lambda
}_{1})Z_{j}(d\boldsymbol{\lambda }_{2})Z_{k}(d\boldsymbol{\lambda
}_{3})Z_{k}(d\boldsymbol{\lambda
}_{4})}{\prod_{i=1}^{4}\|\boldsymbol{\lambda }_{i}\|^{(d-\alpha
)/2}}-\sum_{k=1}^{r}\int_{\mathbb{R}^{4d}}^{\prime \prime }K\left(
\sum_{i=1}^{4 }\boldsymbol{\lambda}_{i}
,\mathcal{D}\right)\frac{\prod_{i=1}^{4}Z_{k}(d\boldsymbol{\lambda
}_{i})}{\prod_{i=1}^{4}\|\boldsymbol{\lambda }_{i}\|^{(d-\alpha
)/2}}\right]\right]^{2}\nonumber\\
& &=\left[\frac{1}{4}\left[r\left(\frac{r}{2}+1\right)\right]^{-1/2}\frac{|\mathcal{D}|}{[\nu (\alpha )]^{2}}\right]^{2}E\left[Y_{1T}-Y_{1}+Y_{2}-Y_{2T}\right]^{2}\nonumber\\
& &\leq \left[\frac{1}{4}\left[r\left(\frac{r}{2}+1\right)\right]^{-1/2}\frac{|\mathcal{D}|}{[\nu(\alpha )]^{2}}\right]^{2}
\left[\left(E\left[Y_{1T}-Y_{1}\right]^{2}\right)^{1/2}+\left(E\left[Y_{2}-Y_{2T}\right]^{2}\right)^{1/2}\right]^{2},\nonumber\\
\label{eqss}
\end{eqnarray}

\noindent where \begin{eqnarray}
E\left[Y_{1T}-Y_{1}\right]^{2}&=&\sum_{k,j;
k\neq j}^{r}E\left|\int_{\mathbb{R}^{2d}}^{\prime  }\int_{\mathbb{R}^{2d}}^{\prime  }K\left(
\sum_{i=1}^{4 }\boldsymbol{\lambda}_{i}
,\mathcal{D}\right)\left[\frac{[\nu (\alpha )]^{2}}{d_{T}}\prod_{i=1}^{4}\sqrt{f_{0}(\|\boldsymbol{\lambda}_{i}\|/T)}\right.\right.\nonumber\\
& &\left.\left. \hspace*{1cm} -\frac{1}{\prod_{i=1}^{4}\|\boldsymbol{\lambda }_{i}\|^{(d-\alpha
)/2}}\right] Z_{j}(d\boldsymbol{\lambda
}_{1})Z_{j}(d\boldsymbol{\lambda }_{2})Z_{k}(d\boldsymbol{\lambda
}_{3})Z_{k}(d\boldsymbol{\lambda
}_{4})\right|^{2}\nonumber\\
&=&\sum_{k,j;
k\neq j}^{r}\int_{\mathbb{R}^{4d}}\left|K\left(
\sum_{i=1}^{4 }\boldsymbol{\lambda}_{i}
,\mathcal{D}\right)\right|^{2}Q_{T}(\boldsymbol{\lambda
}_{1},\boldsymbol{\lambda
}_{2},\boldsymbol{\lambda
}_{3},\boldsymbol{\lambda
}_{4})\prod_{i=1}^{4}\frac{d\boldsymbol{\lambda
}_{i}}{\|\boldsymbol{\lambda }_{i}\|^{d-\alpha
}}\nonumber\\
E\left[Y_{2}-Y_{2T}\right]^{2}&=&\sum_{k=1}^{r}E\left|\int_{\mathbb{R}^{4d}}^{\prime \prime }K\left(
\sum_{i=1}^{4 }\boldsymbol{\lambda}_{i}
,\mathcal{D}\right)\left[\frac{[\nu (\alpha )]^{2}}{d_{T}}\prod_{i=1}^{4}\sqrt{f_{0}(\|\boldsymbol{\lambda}_{i}\|/T)}-\frac{1}{\prod_{i=1}^{4}\|\boldsymbol{\lambda }_{i}\|^{(d-\alpha
)/2}}\right]\right.\nonumber\\
& &\left. \hspace*{3cm} \times Z_{k}(d\boldsymbol{\lambda
}_{1})Z_{k}(d\boldsymbol{\lambda }_{2})Z_{k}(d\boldsymbol{\lambda
}_{3})Z_{k}(d\boldsymbol{\lambda
}_{4})\right|^{2}\nonumber\\
&=&\sum_{k=1}^{r}\int_{\mathbb{R}^{4d}}\left|K\left(
\sum_{i=1}^{4 }\boldsymbol{\lambda}_{i}
,\mathcal{D}\right)\right|^{2}Q_{T}(\boldsymbol{\lambda
}_{1},\boldsymbol{\lambda
}_{2},\boldsymbol{\lambda
}_{3},\boldsymbol{\lambda
}_{4})\prod_{i=1}^{4}\frac{d\boldsymbol{\lambda
}_{i}}{\|\boldsymbol{\lambda }_{i}\|^{d-\alpha
}},\nonumber\\
\label{eqss2}
\end{eqnarray}
\noindent with $$Q_{T}=\left(\prod_{i=1}^{4}\|\boldsymbol{\lambda }_{i}\|^{(d-\alpha
)/2}\frac{[\nu (\alpha )]^{2}}{d_{T}}\prod_{i=1}^{4}\sqrt{f_{0}(\|\boldsymbol{\lambda}_{i}\|/T)}-1\right)^{2}.$$

The convergence to zero of $Q_{T},$ as $T\rightarrow \infty,$ can be
proved as in Theorem 4.1(ii) in Leonenko, Ruiz-Medina and Taqqu
(2014) (see also Leonenko and Olenko, 2013). Hence, from equations
(\ref{eqss}) and (\ref{eqss2}), as $T\rightarrow \infty,$  \begin{eqnarray}& &
E\left[S_{2,T}-\frac{1}{4}\left[r\left(\frac{r}{2}+1\right)\right]^{-1/2}\frac{|\mathcal{D}|}{[\nu (\alpha )]^{2}}\left[\sum_{k,j; k\neq j}^{r}\int_{\mathbb{R}^{4d}}^{\prime \prime
}K\left( \sum_{i=1}^{4 }\boldsymbol{\lambda}_{i}
,\mathcal{D}\right)
\right.\right.\nonumber\\ & &
\left.\left.\hspace*{0.5cm}\times \frac{Z_{j}(d\boldsymbol{\lambda
}_{1})Z_{j}(d\boldsymbol{\lambda }_{2})Z_{k}(d\boldsymbol{\lambda
}_{3})Z_{k}(d\boldsymbol{\lambda
}_{4})}{\prod_{i=1}^{4}\|\boldsymbol{\lambda }_{i}\|^{(d-\alpha
)/2}}-\sum_{k=1}^{r}\int_{\mathbb{R}^{4d}}^{\prime
\prime }K\left( \sum_{i=1}^{4 }\boldsymbol{\lambda}_{i}
,\mathcal{D}\right)\frac{\prod_{i=1}^{4}Z_{k}(d\boldsymbol{\lambda
}_{i})}{\prod_{i=1}^{4}\|\boldsymbol{\lambda }_{i}\|^{(d-\alpha
)/2}}\right]\right]^{2}\nonumber
\end{eqnarray}

\noindent converges to zero, which implies the convergence in probability and hence, in
distribution sense, as we wanted to prove.
\end{proof}

\section{Series representation in terms of independent random variables}
\label{ISR}
We provide here series representations for the limit random variable $S_{\infty}^{\chi^{2}_{r}}$ obtained when the Laguerre rank equals one,
and for the random variable $S_{\infty }$ obtained when the Laguerre rank equals two.
\begin{theorem}
\label{corr} Assume that the conditions of Propositions \ref{fprop}, \ref{par} and Theorems
\ref{th1} and \ref{par2} hold.
\begin{itemize}
\item[(i)] For the case of Laguerre rank equal to one,  the
limit random variable $S_{\infty }^{\chi^{2}_{r}}$ in Theorem
\ref{th1} admits the following series representation:

\begin{equation}
S_{\infty }^{\chi^{2}_{r}}\underset{d}{=}  -\frac{1}{\sqrt{2r}\nu
(\alpha )}|\mathcal{D}|\sum_{j=1}^{r}\sum_{n=1}^{\infty
}\mu_{n}(\widetilde{\mathcal{H}})(\varepsilon_{jn}^{2}-1)=
\sum_{j=1}^{r}\sum_{n=1}^{\infty
}\lambda_{n}(S_{\infty}^{\chi^{2}_{r}})(\varepsilon_{jn}^{2}-1),  \label{eqlsum}
\end{equation}
\noindent where
 $\nu
(\alpha )$ is given in (\ref{eRc}), $\{\varepsilon_{jn},\ n\geq 1,\
j=1,\dots,r\}$ are independent and identically distributed standard
Gaussian random variables, and
$$\lambda_{n}(S_{\infty}^{\chi^{2}_{r}})=-\frac{1}{\sqrt{2r}\nu
(\alpha )}|\mathcal{D}|\mu_{n}(\widetilde{\mathcal{H}}),$$ \noindent
with $\mu_{n}(\widetilde{\mathcal{H}}),$ $n\geq 1,$ being a decreasing sequence
of non-negative real numbers, which are the eigenvalues of the
self-adjoint Hilbert-Schmidt operator
\begin{equation}
\widetilde{\mathcal{H}}(h)(\boldsymbol{\lambda
}_{1})=\int_{\mathbb{R}^{d}}H_{1}\left( \boldsymbol{\lambda
}_{1}-\boldsymbol{\lambda }_{2}\right) h\left(
\boldsymbol{\lambda }_{2}\right) G_{\alpha }(d\boldsymbol{\lambda }%
_{2}):L^{2}_{G_{\alpha }}\left( \mathbb{R}^{d} \right) \longrightarrow
L_{G_{\alpha }}^{2}\left( \mathbb{R}^{d}\right),  \label{hso}
\end{equation}%
\noindent being
\begin{equation}
G_{\alpha }(d\mathbf{x})=\frac{1}{\|\mathbf{x}\|^{d-\alpha
}}d\mathbf{x}. \label{kernelGG}
\end{equation}
\noindent Here, the symmetric kernel $H_{1}\left(\boldsymbol{\lambda
}_{1}-\boldsymbol{\lambda }_{2}\right)=H(\boldsymbol{\lambda }_{1},
\boldsymbol{\lambda }_{2})=K\left(\sum_{i=1}^{2}\boldsymbol{\lambda }_{i},\mathcal{D}\right),$ with $K$ being as before the characteristic function of
the uniform distribution over the set $\mathcal{D}.$

\item[(ii)] For the case of  Laguerre rank equal to two, the
limit random variable $S_{\infty }$ in Theorem \ref{par2} admits the following series
representation:
\begin{eqnarray}
& & \frac{S_{\infty
}}{\frac{1}{4[\nu(\alpha )]^{2}}\left[r\left(\frac{r}{2}+1\right)\right]^{-1/2}|\mathcal{D}|}
\underset{d}{=}\sum_{n=1}^{\infty }\sum_{p=1}^{\infty }\sum_{q=1}^{\infty
}\mu_{n}(\mathcal{H})\gamma_{pn}\gamma_{qn}\nonumber\\
& & \hspace*{1cm}\times \left[\sum_{k,j: k\neq
j}^{r}(\varepsilon_{j,p,n}^{2}-1)(\varepsilon_{k,q,n}^{2}-1)
-\sum_{k=1}^{r}(\varepsilon_{k,p,n}^{2}-1)(\varepsilon_{k,q,n}^{2}-1)\right],
 \label{serep}
\end{eqnarray}
 \noindent where
 $\{\varepsilon_{j,p,n},\ j=1,\dots,r, \ p\geq 1,\  n\geq 1\}$ are
 independent standard Gaussian random variables, in particular,
$E[\varepsilon_{j,p,n}\varepsilon_{k,q,m}]=\delta_{n,m}\delta_{p,q}\delta_{j,k},$
for every $j,k=1,\dots,r,$ and $n,m,q,p\geq 1.$ Here, $\mu_{n}(\mathcal{H}),$ $n\geq 1,$ are the eigenvalues, arranged in decreasing order of their modulus
magnitude, associated with the eigenvectors $\varphi_{n},$ $n\geq 1,$
of the integral operator   $\mathcal{H}:L^{2}_{G_{\alpha }\otimes G_{\alpha }}\left(
\mathbb{R}^{2d} \right) \longrightarrow L^{2}_{G_{\alpha }\otimes
G_{\alpha }}\left( \mathbb{R}^{2d} \right)$ given by,  for all  $h\in L^{2}_{G_{\alpha }\otimes G_{\alpha }}\left(
\mathbb{R}^{2d} \right),$
\begin{equation}
\mathcal{H}(h)(\boldsymbol{\lambda }_{1}, \boldsymbol{\lambda
}_{2})=\int_{\mathbb{R}^{2d}}K\left(
\sum_{i=1}^{4}\boldsymbol{\lambda }_{i},\mathcal{D}\right) h\left(
\boldsymbol{\lambda }_{3}, \boldsymbol{\lambda }_{4}\right)
G_{\alpha }(d\boldsymbol{\lambda }_{3}) G_{\alpha
}(d\boldsymbol{\lambda }_{4}).\label{hso}
\end{equation}

Additionally, for each $n\geq 1,$  $\gamma_{jn},$ $j\geq 1,$ are the eigenvalues, arranged in decreasing order of their modulus magnitude, associated with the
integral operator on $L^{2}_{G_{\alpha }}(\mathbb{R}^{d})$ defined by kernel $\varphi_{n}(\cdot,\cdot).$
\end{itemize}
\end{theorem}

\begin{proof}
The proof of (i) can be derived as in Corollary 4.1 (see Appendix A) in Leonenko, Ruiz-Medina and Taqqu (2014).

\medskip

\noindent (ii) From Proposition \ref{par}, the operator defined in (\ref{hso})
is a Hilbert-Schmidt operator. Equivalently, the kernel
$$H(\boldsymbol{\lambda }_{1},\boldsymbol{\lambda
}_{2},\boldsymbol{\lambda }_{3}, \boldsymbol{\lambda }_{4})=K\left(
\sum_{i=1}^{4}\boldsymbol{\lambda }_{i},\mathcal{D}\right)$$
 \noindent belongs to the space $L^{2}_{G_{\alpha }\otimes G_{\alpha }\otimes G_{\alpha }\otimes G_{\alpha }}\left( \mathbb{R}^{4d}
 \right).$ Thus, $\mathcal{H}\in \mathcal{S}\left(L^{2}_{G_{\alpha }\otimes G_{\alpha }}\left( \mathbb{R}^{2d}
 \right)\right),$
 where as usual $\mathcal{S}(H)$
denotes the Hilbert space of Hilbert Schmidt operators on the
Hilbert space $H.$
 Hence, it admits a kernel spectral representation in terms of a sequence
of eigenfunctions $\{ \varphi_{n},\ n\geq 1\}\subset
L^{2}_{G_{\alpha }\otimes G_{\alpha }}\left( \mathbb{R}^{2d}
\right),$ and a sequence of associated eigenvalues
$\{\mu_{n}(\mathcal{H}),\ n\geq 1\}.$ That is,
\begin{equation}H(\boldsymbol{\lambda }_{1},\boldsymbol{\lambda
}_{2},\boldsymbol{\lambda }_{3}, \boldsymbol{\lambda
}_{4})=\sum_{n=1}^{\infty
}\mu_{n}(\mathcal{H})\varphi_{n}(\boldsymbol{\lambda
}_{1},\boldsymbol{\lambda }_{2})\varphi_{n}(\boldsymbol{\lambda
}_{3},\boldsymbol{\lambda }_{4}).\label{ohh}\end{equation}

In particular, since, for every $n\geq 1,$  $\varphi_{n}\in
L^{2}_{G_{\alpha }\otimes G_{\alpha }}\left( \mathbb{R}^{2d}
\right),$ then,
$$\int_{\mathbb{R}^{2d}}|\varphi_{n}(\boldsymbol{\lambda}_{1},\boldsymbol{\lambda}_{2})|^{2}\frac{d\boldsymbol{\lambda }_{1}d\boldsymbol{\lambda }_{2}}{\|\boldsymbol{\lambda }_{1}\|^{d-\alpha }
\|\boldsymbol{\lambda }_{2}\|^{d-\alpha }} <\infty,$$ \noindent
which means that
$\varphi_{n}(\boldsymbol{\lambda}_{1},\boldsymbol{\lambda}_{2})$
defines  an integral Hilbert-Schmidt operator $\Upsilon $  on
$L^{2}_{G_{\alpha }}(\mathbb{R}^{d}),$ given by $$\Upsilon
(f)(\boldsymbol{\lambda
}_{1})=\int_{\mathbb{R}^{d}}\varphi_{n}(\boldsymbol{\lambda}_{1},\boldsymbol{\lambda}_{2})f(\boldsymbol{\lambda}_{2})G_{\alpha }(d\boldsymbol{\lambda}_{2}),\quad
\forall f\in L^{2}_{G_{\alpha }}(\mathbb{R}^{d}).$$

Therefore, it admits a spectral kernel representation in
$L^{2}_{G_{\alpha }}(\mathbb{R}^{d}),$ in terms of a sequence of
eigenvalues $\{\gamma_{pn},\ p\geq 1\},$ and an orthonormal system
of eigenfunctions $\{ \phi_{pn},\ p\geq 1\}$ of $L^{2}_{G_{\alpha
}}(\mathbb{R}^{d}),$ of the form
\begin{equation}\varphi_{n}(\boldsymbol{\lambda}_{1},\boldsymbol{\lambda}_{2})\underset{L^{2}_{G_{\alpha }\otimes G_{\alpha }}(\mathbb{R}^{2d})}{=}\sum_{p=1}^{\infty }
\gamma_{pn}\phi_{pn}(\boldsymbol{\lambda}_{1})\phi_{pn}(\boldsymbol{\lambda}_{2}),\label{eqvarphi}
\end{equation}
\noindent for each $n\geq 1,$ where convergence holds in the norm of
the space $L^{2}_{G_{\alpha }\otimes G_{\alpha }}(\mathbb{R}^{2d}).$
Replacing in equation (\ref{ohh}) the functions $\{\varphi_{n}\}_{n=1}^{\infty }$ by
their respective series representations as given in equation
(\ref{eqvarphi}), we obtain
\begin{eqnarray}
& & H(\boldsymbol{\lambda }_{1},\boldsymbol{\lambda
}_{2},\boldsymbol{\lambda }_{3}, \boldsymbol{\lambda
}_{4})\underset{L^{2}_{\otimes ^{4}G_{\alpha }}(\mathbb{R}^{4d})
}{=}\sum_{n=1}^{\infty }\sum_{p=1}^{\infty }\sum_{q=1}^{\infty
}\mu_{n}(\mathcal{H})\gamma_{pn}
\gamma_{qn}\phi_{pn}(\boldsymbol{\lambda}_{1})\phi_{pn}(\boldsymbol{\lambda}_{2})\phi_{qn}(\boldsymbol{\lambda}_{3})\phi_{qn}(\boldsymbol{\lambda}_{4}),\nonumber\\
\label{eqfouthrep}
\end{eqnarray}
\noindent where convergence holds in the norm of the space
$L^{2}_{\otimes ^{4}G_{\alpha }}(\mathbb{R}^{4d}):=L^{2}_{G_{\alpha
}\otimes G_{\alpha }\otimes G_{\alpha }\otimes G_{\alpha }}\left(
\mathbb{R}^{4d} \right),$ since, from equations
(\ref{ohh})--(\ref{eqvarphi}),  considering   Minkowski
inequality, we have

\begin{eqnarray}
& & \left\|H(\cdot,\cdot,\cdot, \cdot)-\sum_{n=1}^{\infty
}\sum_{p=1}^{\infty }\sum_{q=1}^{\infty
}\mu_{n}(\mathcal{H})\gamma_{pn} \gamma_{qn}\phi_{pn}(\cdot)\otimes
\phi_{pn}(\cdot)\otimes\phi_{qn}(\cdot)\otimes\phi_{qn}(\cdot)\right\|^{2}_{L^{2}_{\otimes
^{4}G_{\alpha }}(\mathbb{R}^{4d})}
\nonumber\\
& & = \left\|\sum_{n=1}^{\infty
}\mu_{n}(\mathcal{H})\varphi_{n}(\cdot,\cdot )\otimes
\varphi_{n}(\cdot,\cdot)\right.
\nonumber\\
& & \hspace*{2cm}\left.-\sum_{n=1}^{\infty }\sum_{p=1}^{\infty
}\sum_{q=1}^{\infty }\mu_{n}(\mathcal{H})\gamma_{pn}
\gamma_{qn}\phi_{pn}(\cdot)\otimes
\phi_{pn}(\cdot)\otimes\phi_{qn}(\cdot)\otimes\phi_{qn}(\cdot)\right\|^{2}_{L^{2}_{\otimes
^{4}G_{\alpha }}(\mathbb{R}^{4d})}
\nonumber\\
& & \leq \left[\sum_{n=1}^{\infty
}\mu_{n}(\mathcal{H})\left[\int_{\mathbb{R}^{2d}}\left|\varphi_{n}(\boldsymbol{\lambda}_{3},\boldsymbol{\lambda}_{4})\right|^{2}G_{\alpha
}(d\boldsymbol{\lambda}_{3})G_{\alpha
}(d\boldsymbol{\lambda}_{4})\right]^{1/2}\right.
\nonumber\\
& & \hspace*{0.5cm}\left.\times  \left[\int_{\mathbb{R}^{2d}}
\left|\varphi_{n}(\boldsymbol{\lambda}_{1},\boldsymbol{\lambda}_{2})-\sum_{p=1}^{\infty }\gamma_{pn}\phi_{pn}(\boldsymbol{\lambda}_{1})\phi_{pn}(\boldsymbol{\lambda}_{2})\right|^{2}
G_{\alpha }(d\boldsymbol{\lambda}_{1})G_{\alpha
}(d\boldsymbol{\lambda}_{2})\right]^{1/2}\right]^{2}=0,\nonumber\\
 \label{eqfouthrep2}
\end{eqnarray}
\noindent where  we have applied convergence in $L^{2}_{G_{\alpha
}\otimes G_{\alpha }}(\mathbb{R}^{2d})$ of the series
$\sum_{q=1}^{\infty } \gamma_{qn}\phi_{qn}(\cdot )\otimes
\phi_{qn}(\cdot )$ to the function $\varphi_{n}(\cdot,\cdot ),$ for
each $n\geq 1,$  which, in particular, implies that  such a series
differs from $\varphi_{n}(\cdot,\cdot )$ in a set of  null
$G_{\alpha }\otimes G_{\alpha }$-measure.

From Theorem \ref{par2},
\begin{eqnarray}
& & S_{\infty }\underset{d}{=}
\frac{1}{4[\nu (\alpha )]^{2}}\left[r\left(\frac{r}{2}+1\right)\right]^{-1/2}|\mathcal{D}|\left[\sum_{k,j;
k\neq j}^{r}\int_{\mathbb{R}^{2d}}^{\prime }\int_{\mathbb{R}^{2d}}^{\prime }K\left(
\sum_{i=1}^{4 }\boldsymbol{\lambda}_{i}
,\mathcal{D}\right)\right.\nonumber\\
& & \hspace*{5.2cm}\left.\times \frac{Z_{j}(d\boldsymbol{\lambda
}_{1})Z_{j}(d\boldsymbol{\lambda }_{2})Z_{k}(d\boldsymbol{\lambda
}_{3})Z_{k}(d\boldsymbol{\lambda
}_{4})}{\prod_{i=1}^{4}\|\boldsymbol{\lambda }_{i}\|^{(d-\alpha
)/2}}\right.\nonumber\\ & &
\left.\hspace*{3.5cm}-\sum_{k=1}^{r}\int_{\mathbb{R}^{4d}}^{\prime
\prime }K\left( \sum_{i=1}^{4 }\boldsymbol{\lambda}_{i}
,\mathcal{D}\right)\frac{\prod_{i=1}^{4}Z_{k}(d\boldsymbol{\lambda
}_{i})}{\prod_{i=1}^{4}\|\boldsymbol{\lambda }_{i}\|^{(d-\alpha
)/2}}\right].\nonumber
\end{eqnarray}
 \noindent Replacing
$H(\boldsymbol{\lambda }_{1},\boldsymbol{\lambda
}_{2},\boldsymbol{\lambda }_{3}, \boldsymbol{\lambda }_{4})=K\left(
\sum_{i=1}^{4 }\boldsymbol{\lambda}_{i} ,\mathcal{D}\right)$ by its
series representation  (\ref{eqfouthrep}) in the above equation, one
obtains
\begin{eqnarray}
& & S_{\infty }\underset{d}{=}
\frac{|\mathcal{D}|}{4[\nu (\alpha )]^{2}}\left[r\left(\frac{r}{2}+1\right)\right]^{-1/2}
\left[\sum_{k,j: k\neq j}^{r}\sum_{n=1}^{\infty }\sum_{p=1}^{\infty
}\sum_{q=1}^{\infty
}\mu_{n}(\mathcal{H})\gamma_{pn}\gamma_{qn}\int_{\mathbb{R}^{2d}}^{\prime }\prod_{i=1}^{2}\phi_{pn}(\boldsymbol{\lambda}_{i})\frac{Z_{j}(d\boldsymbol{\lambda}_{i})}{\|\boldsymbol{\lambda}_{i}\|^{(d-\alpha
)/2}}\right.\nonumber\\
& &\left. \hspace*{2cm}\times
\int_{\mathbb{R}^{2d}}^{\prime }\prod_{i=3}^{4}\phi_{qn}(\boldsymbol{\lambda}_{i})\frac{Z_{k}(d\boldsymbol{\lambda}_{i})}{\|\boldsymbol{\lambda}_{i}\|^{(d-\alpha
)/2}}- \sum_{k=1}^{r}\sum_{n=1}^{\infty }\sum_{p=1}^{\infty
}\sum_{q=1}^{\infty
}\mu_{n}(\mathcal{H})\gamma_{pn}\gamma_{qn}\right.\nonumber\\
& & \left.\hspace*{2cm}\times
\int_{\mathbb{R}^{2d}}^{\prime }\prod_{i=1}^{2}\phi_{pn}(\boldsymbol{\lambda}_{i})\frac{Z_{k}(d\boldsymbol{\lambda}_{i})}{\|\boldsymbol{\lambda}_{i}\|^{(d-\alpha
)/2}}
\int_{\mathbb{R}^{2d}}^{\prime }\prod_{i=3}^{4}\phi_{qn}(\boldsymbol{\lambda}_{i})\frac{Z_{k}(d\boldsymbol{\lambda}_{i})}{\|\boldsymbol{\lambda}_{i}\|^{(d-\alpha
)/2}}\right]\nonumber
\end{eqnarray}
\begin{eqnarray}
& &
\underset{d}{=}\frac{|\mathcal{D}|}{4[\nu (\alpha )]^{2}}\left[r\left(\frac{r}{2}+1\right)\right]^{-1/2}
\sum_{n=1}^{\infty }\sum_{p=1}^{\infty }\sum_{q=1}^{\infty
}\mu_{n}(\mathcal{H})\gamma_{pn}\gamma_{qn}\left[ \sum_{k,j: k\neq
j}^{r}\int_{\mathbb{R}^{2d}}^{\prime }\prod_{i=1}^{2}\phi_{pn}(\boldsymbol{\lambda}_{i})\frac{Z_{j}(d\boldsymbol{\lambda}_{i})}{\|\boldsymbol{\lambda}_{i}\|^{(d-\alpha
)/2}}\right.\nonumber\\
& & \hspace*{2cm}\left.\times
\int_{\mathbb{R}^{2d}}^{\prime }\prod_{i=3}^{4}\phi_{qn}(\boldsymbol{\lambda}_{i})\frac{Z_{k}(d\boldsymbol{\lambda}_{i})}{\|\boldsymbol{\lambda}_{i}\|^{(d-\alpha
)/2}}-\sum_{k=1}^{r}\int_{\mathbb{R}^{2d}}^{\prime }\prod_{i=1}^{2}\phi_{pn}(\boldsymbol{\lambda}_{i})\frac{Z_{k}(d\boldsymbol{\lambda}_{i})}{\|\boldsymbol{\lambda}_{i}\|^{(d-\alpha
)/2}}\right.\nonumber\\
& &\hspace*{7cm} \left.\times
\int_{\mathbb{R}^{2d}}^{\prime }\prod_{i=3}^{4}\phi_{qn}(\boldsymbol{\lambda}_{i})\frac{Z_{k}(d\boldsymbol{\lambda}_{i})}{\|\boldsymbol{\lambda}_{i}\|^{(d-\alpha
)/2}}\right].\nonumber
\end{eqnarray}
\noindent Applying It$\hat{o}$'s formula (see, for example, Dobrushin and
Major, 1979; Major, 1981),
\begin{eqnarray}
& & S_{\infty }\underset{d}{=}
\frac{|\mathcal{D}|}{4[\nu (\alpha )]^{2}}\left[r\left(\frac{r}{2}+1\right)\right]^{-1/2}
\sum_{n=1}^{\infty }\sum_{p=1}^{\infty }\sum_{q=1}^{\infty
}\mu_{n}(\mathcal{H})\gamma_{pn}\gamma_{qn} \nonumber\\
& & \hspace*{3cm}\times \left[\sum_{k,j: k\neq j}^{r}
H_{2}\left(\int_{\mathbb{R}^{d}}^{\prime }\phi_{pn}(\boldsymbol{\lambda })\frac{Z_{j}(d\boldsymbol{\lambda })}{\|\boldsymbol{\lambda }\|^{(d-\alpha
)/2}}\right)
H_{2}\left(\int_{\mathbb{R}^{d}}^{\prime }\phi_{qn}(\boldsymbol{\lambda })\frac{Z_{k}(d\boldsymbol{\lambda })}{\|\boldsymbol{\lambda }\|^{(d-\alpha
)/2}}\right)\right.\nonumber\\ & & \left.-
\sum_{k=1}^{r}H_{2}\left(\int_{\mathbb{R}^{d}}^{\prime }\phi_{pn}(\boldsymbol{\lambda })\frac{Z_{k}(d\boldsymbol{\lambda })}{\|\boldsymbol{\lambda }\|^{(d-\alpha
)/2}}\right)
H_{2}\left(\int_{\mathbb{R}^{d}}^{\prime }\phi_{qn}(\boldsymbol{\lambda })\frac{Z_{k}(d\boldsymbol{\lambda })}{\|\boldsymbol{\lambda }\|^{(d-\alpha
)/2}}\right)\right]\underset{d}{=}\frac{|\mathcal{D}|\left[r\left(\frac{r}{2}+1\right)\right]^{-1/2}}{4[\nu (\alpha )]^{2}}\nonumber\\
& &\times \sum_{n=1}^{\infty }\sum_{p=1}^{\infty }\sum_{q=1}^{\infty
}\mu_{n}(\mathcal{H})\gamma_{pn}\gamma_{qn}\left[\sum_{k,j: k\neq
j}^{r}(\varepsilon_{j,p,n}^{2}-1)(\varepsilon_{k,q,n}^{2}-1)
-\sum_{k=1}^{r}(\varepsilon_{k,p,n}^{2}-1)(\varepsilon_{k,q,n}^{2}-1)\right]\nonumber\\
\label{sef}
\end{eqnarray}
\noindent as we wanted to prove.

In addition, the orthonormality of
the systems of eigenfunctions $\{ \varphi_{n},\ n\geq 1\},$ in the
Hilbert space
 $L^{2}_{G_{\alpha }\otimes
G_{\alpha }}(\mathbb{R}^{2d}),$ means that

\begin{equation}\left\langle \varphi_{n},\varphi_{k}\right\rangle_{L^{2}_{G_{\alpha }\otimes G_{\alpha }}(\mathbb{R}^{2d})}=\int_{\mathbb{R}^{2d}}\varphi_{n}(\boldsymbol{\lambda }_{1},\boldsymbol{\lambda }_{2})\varphi_{k}(\boldsymbol{\lambda }_{1},\boldsymbol{\lambda }_{2})
G_{\alpha }(d\boldsymbol{\lambda }_{1})G_{\alpha }(d\boldsymbol{\lambda }_{2})=\delta_{n,k},\label{eqipo}
\end{equation}
\noindent where, as before $\delta_{n,k}$ denotes the Kronecker delta function. Replacing $\varphi_{n}$ and $\varphi_{k}$ in (\ref{eqipo}) by its
series representation (\ref{eqvarphi}) in $L^{2}_{G_{\alpha }}(\mathbb{R}^{d}),$ we obtain
\begin{eqnarray}& & \hspace*{0.5cm} \left\langle \varphi_{n},\varphi_{k}\right\rangle_{L^{2}_{G_{\alpha }\otimes G_{\alpha }}(\mathbb{R}^{2d})}=\delta_{n,k}\nonumber\\
& & =\sum_{p=1}^{\infty}\sum_{q=1}^{\infty }\gamma_{pn}\gamma_{qk}\left[\int_{\mathbb{R}^{d}}\phi_{pn}(\boldsymbol{\lambda }_{1})\phi_{qk}(\boldsymbol{\lambda }_{1})G_{\alpha }(d\boldsymbol{\lambda }_{1})\right]\left[\int_{\mathbb{R}^{d}}\phi_{pn}(\boldsymbol{\lambda }_{2})\phi_{qk}(\boldsymbol{\lambda }_{2})G_{\alpha }(d\boldsymbol{\lambda }_{2})\right]\nonumber\\
& & =\sum_{p=1}^{\infty}\sum_{q=1}^{\infty }\gamma_{pn}\gamma_{qk}\left[\left\langle \phi_{pn},\phi_{qk}\right\rangle_{L^{2}_{G_{\alpha }}(\mathbb{R}^{d})}\right]^{2},\label{eqinnpoef}
\end{eqnarray}
\noindent which implies that \begin{equation}\left\langle
\phi_{pn},\phi_{qk}\right\rangle_{L^{2}_{G_{\alpha
}}(\mathbb{R}^{d})}=0,\quad n\neq k,\quad \forall p,q\geq
1,\label{eqipobb0}
\end{equation}
\noindent and  we know that
\begin{equation}\left\langle  \phi_{pn},\phi_{qk}\right\rangle_{L^{2}_{G_{\alpha }}(\mathbb{R}^{d})}=\delta_{p,q},\quad n=k,\label{eqipobb}
\end{equation}
\noindent from the orthonormality of the system of eigenfunctions providing the diagonal spectral representation (\ref{eqvarphi}) of $\varphi_{n},$ for
each $n\geq 1.$ Hence, in addition, from (\ref{eqinnpoef}) and (\ref{eqipobb}), we have $$\sum_{p=1}^{\infty }\gamma_{pn}^{2}=1,\quad \forall n\geq 1.$$
\noindent Thus,  from equations (\ref{sef})--(\ref{eqipobb}), and from the independence of the Gaussian copies
  $Y_{i}(\cdot ),$ $i=1,\dots,r,$ of random field $Y(\cdot),$ we obtain
  $E[\varepsilon_{j,q,n}\varepsilon_{k,p,m}]=\delta_{n,m}\delta_{p,q}\delta_{k,j},$
for every $j,k=1,\dots,r,$ and $n,m,q,p\geq 1.$

\end{proof}
\begin{corollary}
\label{corf} Under the conditions of Theorem \ref{corr}, for
Laguerre rank equal to two,
\begin{eqnarray}
S_{\infty }&\underset{d}{=}&\frac{|\mathcal{D}|}{4[\nu (\alpha )]^{2}}\left[r\left(\frac{r}{2}+1\right)\right]^{-1/2}
\sum_{n=1}^{\infty }\mu_{n}(\mathcal{H})\boldsymbol{\eta}_{n}\nonumber\\
&\underset{d}{=}&\frac{|\mathcal{D}|}{4[\nu (\alpha
)]^{2}}\left[r\left(\frac{r}{2}+1\right)\right]^{-1/2}
\nonumber\\
& \times &\sum_{n=1}^{\infty
}\mu_{n}(\mathcal{H})\left[\sum_{p=1}^{\infty }\sum_{q=1}^{\infty
}\gamma_{pn}\gamma_{qn}
\mathbf{1}^{T}\left(\boldsymbol{\varepsilon}_{p,n}\otimes
\boldsymbol{\varepsilon}_{q,n}-\mbox{Trace}\left(\boldsymbol{\varepsilon}_{p,n}\otimes
\boldsymbol{\varepsilon}_{q,n}\right)\right)\mathbf{1}-\mbox{Trace}\left(\boldsymbol{\varepsilon}_{p,n}\otimes
\boldsymbol{\varepsilon}_{q,n}\right)\right],
\nonumber\\
\label{expmat}
\end{eqnarray}
\noindent where  $\mathbf{1}^{T}$ and $\mathbf{1}$ respectively  are
$1\times r$ and $r\times 1$ vectors with entries equal to one. For
$p,n\geq 1,$ $\boldsymbol{\varepsilon}_{p,n}$ denotes a $r\times 1$
random vector with entries $\varepsilon_{i,p,n}^{2}-1,$
$i=1,\dots,r,$ $\otimes$ denotes the tensorial product of vectors,
and $\mbox{Trace}(\mathbf{A})$ the trace of a matrix $\mathbf{A}.$
Thus, $S_{\infty }$ admits an infinite series representation in
terms of the sequence of independent random variables
$\{\boldsymbol{\eta}_{n},\ n\geq 1\}$ given in (\ref{expmat}).
\end{corollary}

The $\{\boldsymbol{\eta}_{n},\ n\geq 1\}$ are independent because,
for each $n\geq 1,$ $\boldsymbol{\eta}_{n}$ is a function of random variables
$\{(\varepsilon_{i,p,n}^{2}-1), \ i=1,\dots,r,\  p\geq 1\}$ and, as
follows from their definition in equation (\ref{sef}), for $n\neq
k,$ with $n,k\geq 1,$ $\{(\varepsilon_{i,p,n}^{2}-1), \
i=1,\dots,r,\  p\geq 1\}$  and $\{(\varepsilon_{i,q,k}^{2}-1), \
i=1,\dots,r,\  q\geq 1\}$ are mutually independent,  since the
function sequences $\{ \phi_{pn}\}_{p\geq 1}$ and $\{
\phi_{qk}\}_{q\geq 1}$
 are orthogonal in the space $L^{2}_{G_{\alpha }}(\mathbb{R}^{d})$
 (see equation (\ref{eqipobb0})).
\section{Infinite divisibility}
\label{FC}
 When the Laguerre rank equals  one, $S_{\infty }^{\chi^{2}_{r}}$ is infinitely divisible. Theorem \ref{corr}(i) allows the derivation of its L{\'e}vy-Khintchine representation. It is given by
 \begin{theorem}
Under the conditions of Proposition  \ref{fprop} and Theorem
\ref{th1},
 
\begin{equation}
\phi (\theta )=E\left[\exp\left(\mathrm{i}\theta S_{\infty
}^{\chi^{2}_{r}}\right)\right]=\exp \left( \int_{0}^{\infty }\left( \exp
(\mathrm{i}u\theta )-1-\mathrm{i}u\theta \right) \mu _{\alpha
/d}(du)\right), \label{eq45}
\end{equation}%
\noindent where $\mu _{\alpha /d}$ is supported on $(0,\infty )$
having density
\begin{equation}
q_{\alpha /d}(u)=\frac{r}{2u}\sum_{k=1}^{\infty }\exp \left( -\frac{u}{%
2\lambda _{k}(S_{\infty }^{\chi^{2}_{r}})}\right) ,\quad u>0.  \label{levydensity}
\end{equation}%
Furthermore, $q_{\alpha /d}$ has the following asymptotics as $%
u\longrightarrow 0^{+}$ and $u\longrightarrow \infty ,$
\begin{eqnarray}
&&q_{\alpha /d}(u)\sim \frac{[\widetilde{c}(d,\alpha )|\mathcal{D}
|^{(d-\alpha )/d}]^{1/(1-\alpha /d)}\Gamma \left( \frac{1}{1-\alpha
/d}\right)
\left( \frac{u}{2}\right) ^{-1/(1-\alpha /d)}}{2u[(1-\alpha /d)]}  \notag \\
&=&\frac{2^{\frac{\alpha /d}{1-\alpha /d}}[\widetilde{c}(d,\alpha )|\mathcal{D} |^{(d-\alpha )/d}]^{1/(1-\alpha /d)}\Gamma \left( \frac{1}{1-\alpha /d}%
\right) u^{\frac{(\alpha /d)-2}{(1-\alpha /d)}}}{[(1-\alpha
/d)]}\quad \mbox{as}\
u\longrightarrow 0^{+},  \notag \\
&&q_{\alpha /d}(u)\sim \frac{r}{2u}\exp (-u/2\lambda _{1}(S_{\infty
}^{\chi^{2}_{r}})),\quad \mbox{as}\ u\longrightarrow \infty ,  \label{eqDCT2th}
\end{eqnarray}%
\noindent where \begin{equation*}
\widetilde{c}(d,\alpha )= \pi^{\alpha
/2}\left(\frac{2}{d}\right)^{(d-\alpha )/d}\frac{\Gamma
\left(\frac{d-\alpha}{2}\right)}{\Gamma \left(\frac{\alpha
}{2}\right)\left[\Gamma
\left(\frac{d}{2}\right)\right]^{(d-\alpha)/d}}.
  \label{eqctilde}
\end{equation*}
\end{theorem}
\begin{proof}
Let us first consider a truncated version of
the random series representation (\ref{eqlsum})
$$S_{\infty }^{(M)}=\sum_{l=1}^{r}\sum_{k=1}^{M}\lambda_{k}(S_{\infty }^{\chi^{2}_{r}})(\varepsilon_{lk}^{2}-1),$$
\noindent with $S_{\infty }^{M}\underset{d}{\longrightarrow
}S_{\infty}^{\chi^{2}_{r}},$  as $M$ tends to infinity. From the L\'evy-Khintchine
representation of the chi-square distribution (see, for instance,
Applebaum, 2004, Example 1.3.22),
\begin{eqnarray}
& &E\left[\exp(\mathrm{i}\theta
S_{\infty}^{(M)})\right]=\prod_{l=1}^{r}\prod_{k=1}^{M}E\left[\exp\left(
\mathrm{i}\theta\lambda_{k}(S_{\infty }^{\chi^{2}_{r}})(\varepsilon_{lk}^{2}-1)
\right)\right]\nonumber\\
& &=\prod_{l=1}^{r}\prod_{k=1}^{M}\exp\left( -\mathrm{i}\theta
\lambda_{k}(S_{\infty }^{\chi^{2}_{r}})
+\int_{0}^{\infty }(\exp(\mathrm{i}\theta
u)-1)\left[
\frac{\exp\left(-u/(2\lambda_{k}(S_{\infty }^{\chi^{2}_{r}}))\right)}{2u}\right]du\right)\nonumber\\
& &=\prod_{k=1}^{M}\exp\left(  r\int_{0}^{\infty
}(\exp(\mathrm{i}\theta u)-1-\mathrm{i}\theta
u)\left[\frac{\exp(-u/2\lambda_{k}(S_{\infty }^{\chi^{2}_{r}}))}{2u}\right]du\right)\nonumber\\
& &=\exp\left(  r\int_{0}^{\infty }(\exp(\mathrm{i}\theta
u)-1-\mathrm{i}\theta u)\left[\frac{1}{2u}G_{\lambda (\alpha /d
)}^{(M)}\left(\exp(-u/2)\right)\right]du\right).\nonumber\\ \label{integrand}
\end{eqnarray}

To apply the Dominated Convergence Theorem, the following upper
bound is used:
\begin{eqnarray}
\left|(\exp(\mathrm{i}\theta u)-1-\mathrm{i}\theta u)\left[
\frac{r}{2u}G_{\lambda (\alpha /d
)}^{(M)}\left(\exp(-u/2)\right)\right]\right| &\leq &
\frac{r\theta^{2}}{4}uG_{\lambda (\alpha
/d)}^{(M)}\left(\exp(-u/2)\right) \nonumber\\ &\leq &
\frac{r\theta^{2}}{4}uG_{\lambda (\alpha /d
)}\left(\exp(-u/2)\right),\nonumber\\ \label{equbtt}
\end{eqnarray}
\noindent where, as indicated in Veillette and Taqqu (2013), we have
applied the inequality $|\exp(\mathrm{i}z)-1-z|\leq
\frac{z^{2}}{2},$ for $z\in \mathbb{R}.$ The right-hand side of
(\ref{equbtt}) is continuous, for $0<u<\infty,$ and from  Lemma 4.1 of
Veillette and Taqqu (2013) with
 $G_{\lambda (\alpha
/d)}^{(M)}(x)=\sum_{k=1}^{M}x^{[\lambda_{k}(S_{\infty }^{\chi^{2}_{r}})]^{-1}},$
keeping in mind the asymptotic order of eigenvalues of operator
$\mathcal{K}_{\alpha }$ (see, for example, Theorem 3.1(i) by Leonenko,
Ruiz-Medina and Taqqu, 2014), we obtain
\begin{eqnarray}
& &uG_{\lambda (\alpha /d)}\left(\exp(-u/2)\right) \sim
u\exp(-u/2\lambda_{1}(S_{\infty }^{\chi^{2}_{r}})),\quad \mbox{as}\
u\longrightarrow
\infty\nonumber\\
& & uG_{\lambda (\alpha /d
)}\left(\exp(-u/2)\right) \sim  [\widetilde{c}(d,\alpha )|D
|^{1-\alpha /d}]^{1/1-\alpha /d}\frac{u}{(1-\alpha /d)}
\nonumber\\
& &\times \Gamma \left( \frac{1}{1-\alpha /d
}\right)(1-\exp(-u/2))^{-1/(1-\alpha /d)}\sim Cu^{-\frac{\alpha /d}{1-\alpha /d}}\quad \mbox{as}\
u\longrightarrow 0, \label{asympt}
\end{eqnarray}
\noindent for some constant $C.$ Since $0<\frac{\alpha /d}{1-\alpha
/d}<1,$  the right-hand side of
(\ref{asympt}), which  does not depend on $M,$  is integrable on
$(0,\infty ).$ Hence, by the Dominated Convergence Theorem,
\begin{eqnarray}
& & E\left[\exp(\mathrm{i}\theta S_{\infty}^{(M)})\right]\longrightarrow E\left[\exp(\mathrm{i}\theta S_{\infty}^{\chi^{2}_{r}})\right] \nonumber\\
& &=\exp\left( \int_{0}^{\infty }(\exp(\mathrm{i}\theta
u)-1-\mathrm{i}\theta u)\left[\frac{r}{2u}G_{\lambda (\alpha /d
)}\left(\exp(-u/2)\right)\right]du\right), \label{eqDCT}
\end{eqnarray}
 \noindent which proves that equations
(\ref{eq45}) and (\ref{levydensity}) hold. Equation (\ref{eqDCT2th})
follows, in a similar way to the proof of  Theorem 5.1(i) in Leonenko,
Ruiz-Medina and Taqqu (2014), considering the expression obtained
by the  L\'evy density $q$  in equation (\ref{levydensity}).
\end{proof}

\bigskip

From the above equations, in a similar way  as in Theorem 5.1(ii)-(iv) by Leonenko, Ruiz-Medina and Taqqu (2014), it can be seen that
$S_{\infty }^{\chi^{2}_{r}}\in \mathcal{ID}(\mathbb{R})$ is selfdecomposable. Hence, it has a
bounded density. It can also be showed that $S_{\infty }^{\chi^{2}_{r}}$ is in the  Thorin class with
Thorin measure
\begin{equation*}
U(dx)=\frac{r}{2}\sum_{k=1}^{\infty
}\delta_{\frac{1}{2\lambda_{k}(S_{\infty }^{\chi^{2}_{r}})}}(x),
\end{equation*}
\noindent where $\delta_{a}(x)$ is the Dirac delta-function at point
$a.$ Finally, $S_{\infty }^{\chi^{2}_{r}}$
 admits the integral representation
\begin{equation}
S_{\infty}^{\chi^{2}_{r}}\underset{d}{=}\int_{0}^{\infty }
\exp\left(-u\right)d\left(
\sum_{k=1}^{\infty }\lambda_{k}(S_{\infty }^{\chi^{2}_{r}})A^{(k)}(u)\right)\underset{d}{=}%
\int_{0}^{\infty }\exp\left(-u\right)dZ(u),  \label{Thclass}
\end{equation}
\noindent where
\begin{equation}
Z(t)=\sum_{k=1}^{\infty }\lambda_{k}(S_{\infty }^{\chi^{2}_{r}})A^{(k)}(t),\quad
t\geq 0, \label{Thclass2}
\end{equation}
\noindent with $A^{(k)},$ $k\geq 1,$ being independent copies of a
L\'evy process.

{\small
\vspace*{0.5cm} \noindent {\bfseries {\large Acknowledgments}}

This work has been supported in part by project MTM2012-32674 (co-funded with FEDER) of the
DGI, MEC,  Spain. Murad Taqqu was supported in part by the NSF grant DMS--1309009 at Boston University.

\section*{References}

\begin{itemize}
\item[]  Anh, V.V.,  Leonenko, N.N., 1999. Non-Gaussian scenarios for the heat equation with singular initial conditions. \emph{Stochastic Processes and their Applications} \textbf{84}, 91--114.
    \item[]  Anh, V.V.,  Leonenko, N.N., Ruiz-Medina, M.D., 2013. Macroscaling limit theorems for filtered
spatiotemporal random fields. \emph{Stochastic Analysis and Applications} \textbf{31}, 460--508.
\item[]
Applebaum, D., 2004. \emph{L\'evy Processes and Stochastic
Calculus}. Cambridge University Press, Cambridge, UK
\item[] Bateman, H., Erdelyi, A., 1953. \emph{Higher Transcenental Functions}, Vol. II. McGraw-Hill, New York.
\item[]
Berman, S.M., 1982. Local times of stochastic processes with positive definite bivariate densities. \emph{Stochastic Processes and their Applications} \textbf{12}, 1--26.
\item[] Berman, S., 1984. Sojourns of vector Gaussian processes inside and
outside spheres. \emph{Z. Wahrsch. Verw. Gebiete} \textbf{66}, 529--542.
\item[]
 Brelot, M., 1960. \emph{Lectures on Potential Theory}. Tata
Institute of Fundamental Research, Bombay.
\item[]
Chen,  Z.-Q. , Meerschaert, M.M.,   Nane,  E., 2012. Space–time fractional diffusion on bounded domains. \emph{Journal of Mathematical Analysis and
Applications}  \textbf{393},   479--488.
\item[] Dautray, R., Lions, J.L., 1985. \emph{Mathematical Analysis and Numerical Methods for Science and Tecnology}, Vol 3. \emph{Spectral Methods and Applications}.
Springer, New York.
\item[]
Dobrushin, R.L., Major, P., 1979. Non-central limit theorems for
non-linear functionals of Gaussian fields. \emph{Z. Wahrsch. Revw.
Gebiete} \textbf{50}, 1--28.
\item[]
Fuglede, B.  2005.  Dirichlet problems for harmonic maps from regular
domains. \emph{Proc. London Math. Soc.} \textbf{91},   249--272.
\item[]
Fox, R., Taqqu, M.S., 1985. Noncentral limit theorems for
quadratic forms in random variables having long-range dependence.
\emph{Ann. Probab.} \textbf{13}, 428--446.
\item[]
 Gajek, L.,  Mielniczuk, J., 1999. Long- and short-range dependent
sequences under exponential subordination. \emph{Statist. Probab. Letters}
\textbf{43}, 113--121.
\item[] Joe, H., 1997. \emph{Multivariate Models and Dependence Concepts}. Chapman
and Hall, London.
\item[]
Lancaster, H.O., 1958. The structure of bivariate distributions. \emph{Ann. Math. Statist.} \textbf{29}, 719--736.
\item[]
Lancaster, H.O., 1963. Correlations and canonical forms of bivariate distributions. \emph{Ann.
Math. Statistics} \textbf{34},  532--538.
\item[]
Leonenko, N.N., 1999. \emph{Limit Theorems for Random Fields with Singular Spectrum}. Kluwer Academic
Publishers, Dordrecht.
\item[]
Leonenko, N., Olenko, A., 2013. Tauberian and Abelian
theorems for long-range dependent random fields. \emph{Methodlogy and
Computing in Applied Probability} \textbf{15}, 715--742.
\item[]
Leonenko N. and Olenko, A., 2014. Sojourn measures for Student and Fisher-Snedecor random fields. \emph{Bernoulli} \textbf{20},  1454--1483.
\item[]
Leonenko, N.N., Ruiz-Medina, M.D., Taqqu, M., 2014.
Rosenblatt distribution subordinated to gaussian
random fields with long-range dependence, arXiv:submit/1156097.
\item[]
Lukacs, E., 1970. \emph{Characteristic Functions} (Second ed.).
Griffin, London.
\item[]
Major, P., 1981. \emph{Multiple Wiener It$\widehat{\mbox{o}}$ Integrals. Lecture Notes in Mathematics} Vol. 849. Springer, Berlin.
\item[] Mielniczuk, J., 2000. Some properties of random stationary sequences
with bivariate densities having diogonal expansions and parametric
estimations based on them. \textit{Nonparametric Statistics} \textbf{12}, 223--243.

\item[]
Sarmanov, O.V., 1963. Investigation of stationary Markov processes by the method of eigenfunction
expansion. \emph{Translated in Selected Translations in Mathematical Statistics and Probability Theory},
 \emph{Amer. Math. Soc., Providence} \textbf{4},  245--269.
\item[]
Simon, B., 2005. \emph{Trace Ideals and Their Applications.
Mathematical Surveys and Monographs} 120. Providence, RI: American
Mathematical Society (AMS).
\item[]
Stein, E.M., 1970. \emph{Sigular Integrals and Differential
Properties of Functions}. Princenton, University Press, New Jersey.

\item[]  Taqqu, M.S., 1975. Weak-convergence to fractional Brownian
motion and to the Rosenblatt process. \emph{Z. Wahrsch. Verw. Gebiete} \textbf{31},
287--302.

\item[]  Taqqu, M.S., 1979. Convergence of integrated processes of
arbitrary Hermite rank. \emph{Z. Wahrsch. Verw. Gebiete} \textbf{50}, 53--83.

\item[]
Veillette, M.S. and Taqqu, M.S., 2013. Properties and
numerical evaluation of Rosenblatt distribution. \emph{Bernoulli} \textbf{19},
982--1005

\item[]
Wong, E., Thomas, J.B., 1962. On polynomial expansion of second order distributions. \emph{SIAM J. Appl. Math.}
\textbf{10}, 507--516.
\end{itemize}

}
\bigskip

\hspace*{3cm}
{\small
\begin{itemize}
\item[]   N. N. Leonenko\\ Cardiff School of Mathematics, Senghennydd Road, Cardiff CF24 4AG,
United Kingdom\\
E-mail: LeonenkoN@cardiff.ac.uk

\item[] M. D. Ruiz-Medina\\ Department of Statistics and Operations Research, University of
Granada, Campus de Fuente Nueva s/n, E-18071 Granada, Spain.\\
E-mail: mruiz@ugr.es

\item[]
  M. S. Taqqu\\ Department of Mathematics and Statistics, 111 Cummington St., Boston
University, Boston, MA 02215, USA\\
E-mail: murad@bu.edu
\end{itemize}
}

\end{document}